\documentclass[11pt]{article} 

\usepackage{amsmath}
\usepackage{amsfonts}
\usepackage{amsthm}
\usepackage{amssymb}
\usepackage[pdftex]{graphicx}
\usepackage{epstopdf}
\usepackage{psfrag}
\usepackage[usenames,dvipsnames]{color}
\usepackage{thmtools}
\usepackage{thm-restate}
\usepackage{tikz}
\usepackage[numbers,sort&compress]{natbib}
\usepackage{enumitem}

\usepackage{tikz}
\usetikzlibrary{positioning}
\usetikzlibrary{arrows.meta}
\usetikzlibrary{decorations.pathreplacing}
\usetikzlibrary{calc}

\usepackage{scalerel}[2016-12-29]

\def\scaleint#1{\vcenter{\hbox{\scaleto{\displaystyle\int}{#1}}}}

\usepackage{hyperref}

\advance\topmargin by -\headheight
\advance\topmargin by -\headsep
\evensidemargin=\oddsidemargin
\declaretheorem[name=Theorem, numberwithin=section]{theorem}
\declaretheorem[sibling=theorem]{lemma}
\declaretheorem[sibling=theorem]{proposition}
\declaretheorem[sibling=theorem]{corollary}

\declaretheorem[sibling=theorem, style=definition, qed=\ensuremath{\blacktriangle}]{definition}
\declaretheorem[sibling=theorem, style=definition, qed=\ensuremath{\blacktriangle}]{example}

\newcommand{\interior}[1]{\ensuremath{\text{int}\left({#1}\right)}}
\newcommand{\range}[1]{\ensuremath{\operatorname{ran}\left({#1}\right)}}

\newcommand{\grad}{\ensuremath{\nabla}}

\newcommand{\R}{\ensuremath{\mathbb{R}}}
\newcommand{\N}{\ensuremath{\mathbb{N}}}

\newcommand{\dd}{\ensuremath{\textrm{d}}}

\newcommand{\future}[1]{\ensuremath{I^+\left(#1\right)}}
\newcommand{\jfuture}[1]{\ensuremath{J^+\left(#1\right)}}
\newcommand{\jpast}[1]{\ensuremath{J^-\left(#1\right)}}

\newcommand{\past}[1]{\ensuremath{I^-\left(#1\right)}}

\newcommand{\EmptySet}{\ensuremath{\varnothing}}
\newcommand{\esssup}{\ensuremath{{\rm ess\,sup}}}

\newcommand{\norm}[1]{\ensuremath{\left\lVert{#1}\right\rVert}}
\newcommand{\abs}[1]{\ensuremath{\left\lvert{#1}\right\rvert}}

\title{Limit curve theorems for incomplete metric spaces and 
the null distance on Lorentzian manifolds}
\author{Adam Rennie\dag\thanks{email: 
\texttt{renniea@uow.edu.au}, \texttt{ben@benwhale.com}
}
, Ben E. Whale\dag
\\[3pt]
\dag School of Mathematics and Applied Statistics, University of Wollongong\\
Wollongong, Australia\\[3pt]
}

\begin{document}
\maketitle

\begin{abstract}
  We prove a limit curve theorem for incomplete metric spaces. Our main
  application is to Sormani and Vegas' null distance, where our results give
  strong control on the Lorentzian lengths of limit curves.  We also show that
  regular cosmological time functions and the surface function of a Cauchy
  surface in a globally hyperbolic manifold define such a null distance.
\end{abstract}

\tableofcontents

\parindent=0.0in
\parskip=0.06in

\section{Introduction}

  This paper proves limit curve
  theorems for incomplete distances on Lorentzian manifolds
  and establishes new sufficient conditions for the upper semi-continuity
  of the Lorentzian length functional on curves.

  There are natural situations in which one wants the Lorentzian
  length functional to be upper semi-continuous over sequences of curves
  with common non-compact domain. We give necessary conditions
  for this in Lemma \ref{lem_improved_length_bound}.
  Existing
  limit curve theorems, however, assume a complete metric which
  may have little relation to the Lorentzian distance.
  The
  compatibility of the null distance with the Lorentzian distance makes it an
  attractive alternative,
  but there are no guarantees of completeness for the null
  distance. This was a strong motivator for our limit curve theorem,
  which allows 
  for incomplete distances,
  Theorem \ref{thm_CurveUniformConvergenceInBoundedRegion_extension}.

  Theorem \ref{thm_CurveUniformConvergenceInBoundedRegion_extension}
  only requires the metric to induce the manifold topology and need
  not be metrically complete.
  In Section 4 we specialise our limit curve theorem to the null distance
  of Sormani and Vega \cite{sormani2016null}
  so that we can make use of
  Lemma \ref{lem_improved_length_bound}.
  We require additional conditions given in 
  Proposition \ref{lem_cauchy_sort_of}
  and Theorem \ref{thm_reparametrisation_to_control_limsup}.
  We show that the null distance induced by
  a regular cosmological time or a suitable surface function
  \cite{rennie2016generalised} satisfy our conditions.

  The strangest feature of our definitions and results is that limit curves can
  have strictly smaller domains than the curves in the defining sequence. This is
  a direct consequence of using incomplete metrics, because portions of the
  ``obvious'' limit curve may not exist, or be obstructed by an incompleteness.
  See Sections 
  \ref{sec_llc}
  and
  \ref{sec_what_goes_wrong}
  for details and examples.

  In Section 
  \ref{sec_background}
  we recall the required
  background results, including the null distance,
  surface functions and known limit curve theorems. 
  We also study the
  relationship of the Lorentzian length to limit curve theorems, how
  incompleteness damages these relations, and provide our definition of limit
  curve in incomplete spaces.


  Section 
  \ref{sec_general_limit_curve}
  proves our limit curve theorems for incomplete distances, and in
  Section 
  \ref{sec_length_control}
  we exploit null distances associated to (suitable) time functions
  to control the Lorentzian length of limit curves. We show that regular
  cosmological time functions satisfy our requirements. In Section 
  \ref{subsec_surface_function_generalities},
  we show
  that surface functions of $C^1$ Cauchy surfaces give null distances with
  strong control on the Lorentzian lengths of limit
  curves induce null distances which can be used with
  Theorem \ref{thm_CurveUniformConvergenceInBoundedRegion_extension},
  see
  Corollaries \ref{corl:surf}
  and
  \ref{corl:surf2}

  {\bf Acknowledgements} We would like to thank Narla and the other staff at the
  Alabama Hotel, Hobart, where part of this work was conducted.

\section{Background}
  \label{sec_background}

  The material below recalls ideas relied on in this paper.
  General references for the needed Lorentzian geometry 
  are \cite{beem1996global} and
  \cite{penrose1972techniques}. A review of
  material relating to generalised time functions can be found
  in \cite{minguzzi2019lorentzian, minguzzi2008causal}.
  A detailed introduction to the limit curve theorem
  can be found in \cite{minguzzi2008limit}.

  \subsection{Some Lorentzian geometry}

    Our manifolds, denoted by $M$, are Hausdorff, paracompact and smooth.
    We allow an arbitrary number $\geq 1$ of spacelike
    dimensions. The metric is $g$ with signature $(-,+,\ldots, +)$
    and assumed to be smooth. 
    Since $M$ 
    is separable, any compact exhaustion of $M$ is independent of the
    choice of
    any metric or distance.

    We work with continuous curves in $M$.
    A curve is a continuous function
    $\gamma:I\to M$ where $I\subset \R$ is a connected interval.
    In particular, $\gamma(I)$ is connected.
    A change of parameter is a continuous, bijective, strictly
    increasing or decreasing function $s:J\to I$, where
    $J\subset\R$ is a connected interval.

    We extend 
    Penrose' definition of causal curves, 
    \cite[Definition 2.25]{penrose1972techniques},
    to curves that may not have compact
    image. 
    \begin{definition}[Causal, timelike and null curves]
      \label{def_causal_curve}
      A continuous function
      $\gamma:I\to M$ from a connected, not compact,
      interval $I$ of $\R$ is a future directed causal / timelike / null geodesic
      if for all $K\subset I$ a connected compact subinterval,
      the subcurve
      $\gamma|_K$ is a future directed causal / timelike / null curve
      in the sense of \cite[Definition 2.25]{penrose1972techniques}.
    \end{definition}
    A past directed, continuous causal curve is defined by time duality.
    By a causal curve we will always mean a continuous causal curve.


    Every causal curve $\gamma:I\to M$ 
    has a re-parametrisation $s:J\to I$ so that  
    for any chart $\phi:U\to \R^n$, the curve
    $\phi\circ\gamma\circ s$ is locally Lipschitz with respect to the
    Euclidean distance on $\R^n$, 
    \cite[Page 75ff, Equation 3.14]{beem1996global}. See
    also \cite[Theorem 2.12]{minguzzi2019lorentzian}.
    If $h$ is any Riemannian metric on $M$ then
    as $(\gamma\circ s)'$ exists for almost all $t\in J$, the (extended) number
    \[
      a=\int_J
        \sqrt{h((\gamma\circ s)'(t),(\gamma\circ s)'(t))}\,
        \dd t\ \ \in\ [0,\infty],
    \]
    is well-defined. 
    If the interval $I$ is closed  then
    let $A=[0,a]$, otherwise let $A$ be one of $(0,a)$, $[0,a)$, $(0,a]$
    depending on the closedness of $I$ to the left and right.
    Then there is a change of parameter
    $r:A\to J$ so that
    \[
      h((\gamma\circ s\circ r)'(t),(\gamma\circ s\circ r)'(t))=1,
    \]
    for almost all $t\in A$. 
    We call $r$ the arc-length parametrisation induced by $h$.

    A function is Cauchy if all of its level surfaces
    are Cauchy hypersurfaces. 
    We allow Cauchy functions
    to not be surjective onto $\R$. 
    Thus, we differ from \cite{bernal2003smooth}.

    Penrose has shown that continuous causal curves with compact image
    have well defined Lorentzian length
    \cite[Definition 7.4ff]{penrose1972techniques}.
    We can extend this definition to our continuous curves.
    If $\gamma:I\to M$ is a continuous causal curve with an arbitrary connected
    interval as a domain, we define
    \begin{equation}
      \label{eq_lg_def}
      L(\gamma)=\sup\{L(\gamma|_{[a,b]}:a,b\in I,\, a<b\},
    \end{equation}
    where $L(\gamma|_{[a,b]})$ is the Lorentzian length 
    as defined by Penrose
    \cite[Definition 7.4ff]{penrose1972techniques}.
    Thus, by definition if $\gamma:[0,a)\to M$ is a continuous causal curve then
    $L(\gamma)=\lim_{t\to a}L(\gamma|_{[0,t]})$. 

    A function $f:M\to\R$ is locally Lispschitz if on any compact subset $C$ of
    a chart, $f$ is Lipschitz with respect to any metric inducing the manifold
    topology on the chart.

    When $\gamma:I\to M$ is locally Lipschitz we have
    \begin{equation}
      \label{eq_rect_h1_area}
      L(\gamma)=\int_I\sqrt{-g(\gamma',\gamma')}\dd t,
    \end{equation}
    by \cite[Theorem 2.37]{minguzzi2019causality}.
    Minguzzi presents a similar discussion of
    causal curves with an emphasis on absolute continuity
    \cite[Sections 2.3, 2.4, and 2.5]{minguzzi2019causality}.

  \subsection{Cosmological time, surface functions and the null distance}
    
    The Lorentzian distance will be denoted $d_L$.
    If $h$ is an auxiliary Riemannian metric we write
    $d(\cdot,\cdot; h)$ for the (actual) distance induced by $h$. 
    The subscript
    $L$ on the Lorentzian distance is intended to remind the reader that
    $d_L$ is not a distance. 

    \begin{definition}[Cosmological time, {\cite{andersson1998cosmological}}]
      The cosmological times
      $\tau:M\to\R$ is defined by
      \[
        \tau(x)
        =
        \sup\left\{d_L(y,x): y\in\past{x}\right\}
      \]
      If $\tau(x)<\infty$ for all $x\in M$ 
      (so that the earliest time is a finite time ago) 
      and $\lim_{t}\tau(\gamma(t))=0$ 
      for all past-directed inextendible causal
      curves then we say that $\tau$ is regular.
    \end{definition}

    Given $S\subset M$ we define
    \begin{align*}
      d_L(S,x)&=\sup\{d_L(s,x):s\in S\}
      &\text{and}&&
      d_L(x,S)&=\sup\{d_L(x,s):s\in S\}.
    \end{align*}
    
    \begin{definition}[Surface function]
      \label{def_surface_function}
      If $S\subset M$ is an achronal 
      set such that $M=\future{S}\cup S\cup\past{S}$
      and for all $x\in M$, $d_L(S,x)<\infty$ and $d_L(x,S)<\infty$ then
      the function $\tau_S:M\to\R$ given by
      \[
        \tau_S(x)=\left\{\begin{aligned}
          &d_L(S,x), && x\in\future{S}\\
          &0, && x\in S,\\
          &-d_L(x, S), && x\in\past{S},
        \end{aligned}\right.
      \]
      is well-defined.
      We call $\tau_S$ the surface function associated to $S$.
    \end{definition}

    Surface functions, and in particular suitable surfaces, can be constructed
    on any Lorentzian manifold $(M,g)$ with
    finite Lorentzian distance, meaning that for all $x,y\in M$,
    $d_L(x,y)<\infty$, see \cite{rennie2016generalised}. 
    Surface functions
    are increasing on timelike curves and non-decreasing on null curves. 
    Surface
    functions are continuous almost everywhere on $M$, 
    \cite[Corollary A.2]{rennie2016generalised},
    and are differentiable almost everywhere on $M$,
    \cite[Theorem 1.19]{minguzzi2019causality}.

    We now introduce the null distance of a generalised time function
    \cite{sormani2016null}. 
    A generalised time function is any function that is increasing on
    causal curves.
    An alternating causal curve is a $C^0$ piecewise $C^\infty$
    function $\gamma:[a,b]\to M$, $[a,b]\subset \R$, with
    a finite partition $\{t_1,\ldots,t_k\}$, $[t_1,t_k]=[a,b]$
    so that for all $i=1,\dots, k-1$, $\gamma|_{[t_i,t_{i+1}]}$ is a
    smooth \emph{future or past} directed causal curve.
    Note that we use the phrase ``alternating causal curve'' where
    \cite{sormani2016null} uses the phrase ``piecewise causal curve'', which
    means something different for us.

    Any generalised time function $\tau$
    induces a functional $L(\cdot;\tau)$ on alternating causal curves,
    \begin{definition}[{\cite[Definition 3.2]{sormani2016null}}]
      \label{def_null_distance}
      Let $\tau$ be a generalised time function.
      If $\gamma:[a,b]\to M$, $[a,b]\subset \R$, with
      a finite partition $\{t_1,\ldots,t_k\}$, $[t_1,t_k]=[a,b]$,
      is an alternating causal curve then we define
      $L(\gamma; \tau)=
       \sum_{i=1}^{k-1}\abs{\tau(\gamma(t_{i+1})) - \tau(\gamma(t_i))}$.
      The function $d(\cdot,\cdot\,;\tau):M\times M\to\R$ called the
      null distance of $\tau$, \cite[Definition 3.2]{sormani2016null}, is
      defined by
      \begin{equation*}
        d(x,y;\tau)=\inf\{L(\gamma;\tau): \gamma\
        \text{is an alternating causal curve from}\ x\ \text{to}\ y\}.
        \qedhere
      \end{equation*}
    \end{definition}
    Note that null distances may only be pseudo-distances and
    may not induce the topology of the manifold.
    If $\gamma:[0,1]\to M$ is a future directed, piecewise smooth, 
    causal curve then
    $L(\gamma; \tau)=\tau(\gamma(1))-\tau(\gamma(0))$,
    \cite[Lemma 3.6(2)]{sormani2016null}.
    If $y\in\jfuture{x}$ then $d(x,y;\tau)=\tau(y)-\tau(x)$,
    \cite[Lemma 3.11]{sormani2016null}.

    To describe when a null distance is a distance that
    induces the manifold topology we need to introduce the concept
    of ``anti-Lipschitz''. Anti-Lipschitz functions were first
    introduced in \cite[Lemma 4.1ff]{chrusciel2016differentiability}.
    Anti-Lipschitz functions are also defined in
    \cite[Section 2.2, Item (e), Page 21]{minguzzi2019causality}
    and
    \cite[Definition 4.4]{sormani2016null}. 
    Equivalence of these various definitions is
    proven in \cite[Prop 4.9]{sormani2016null}.

    \begin{definition}
      \label{def_antilip}
      Let $(M,g)$ be a Lorentzian manifold and $f:M\to\R$ a function.
      The function $f$ is anti-Lipschitz 
      on an open set $U\subset M$
      if there exists a Riemannian metric $h:TU\times TU\to M$ so that
      for all $\gamma:[0,1]\to U$ a future directed causal curve
      \[
        f\circ\gamma(1)-f\circ\gamma(0)\geq L(\gamma;h).
      \]
      A function $f$ is locally anti-Lipschitz if it is anti-Lipschitz
      on a neighbourhood of each $p\in M$.
    \end{definition}
    An anti-Lipschitz function $f$ is increasing on all future directed causal
    curves
    and therefore is a generalised time function, 
    \cite[Definition 4.4]{sormani2016null}. We recall the following
    characterisation of when the null distance of a generalised time function
    is a metric compatible with the manifold topology.

    \begin{proposition}\cite[Propositions 3.15, 4.5]{sormani2016null}
      \label{prop:SV-metric}
      Let $f:M\to\R$ be a generalised time function.
      If $f$ is locally anti-Lipschitz then
      $d(\cdot,\cdot; f)$ is a distance.
      If, in addition, $f$ is continuous
      then 
      $d(\cdot,\cdot; f)$ induces the manifold topology.
    \end{proposition}

  \subsection{Smooth approximation of Lipschitz functions}
    \label{app:Lip}
  
    We shall use
    Czarnecki and Rifford's approximation theorem,
    \cite[Theorem 2.2, page 4475]{czarnecki2006approximation},
    to smoothly approximate a locally Lipschitz function.
    This result relies on
    uses Clarke's generalised gradient \cite{clarke1990optimization}.
    We briefly review the required results and definitions here.

    \begin{definition}[Clarke's generalised directional derivative, {\cite[Page 25]{clarke1990optimization}}]
      Let $Y\subset \R^n$, $v, x\in \R^n$.
      If $f:Y\to \R$ is Lipschitz on a neighbourhood of $x$ 
      then
      the generalised directional derivative of $f$ at $x$ in the direction
      $v$ is denoted $f^\circ(x;v)$ and is defined by
      \[
        f^\circ(x;v)=\limsup_{\substack{y\to x\\ t\downarrow 0}}
          \frac{f(y + tv) - f(y)}{t}.\qedhere
      \]
    \end{definition}

    The generalised directional derivative is similar to a partial strong
    derivative but uses a $\limsup$
    rather than $\lim$. 
    See \cite[Proposition 2.1.1ff]{clarke1990optimization} for
    proofs of some properties.

    \begin{definition}[Clarke's generalised gradient, {\cite[Page 27]{clarke1990optimization}}]
      \label{def_cgg}
      Let $Y\subset \R^n$, $v, x\in \R^n$.
      If $f:Y\to \R$ is Lipschitz on a neighbourhood of $x$ 
      then the generalised gradient of $f$ at $x$ is
      denoted $\partial^\circ f(x)$ 
      or
      $\partial^\circ f|_x$ 
      and is defined by
      \[
        \partial^\circ f(x) = 
        \{w\in\R^n: \forall v\in\R^n,\, f^\circ(x;v)\geq w\cdot v\},
      \]
      where $w\cdot v$ is the dot product, the standard Euclidean norm on
      $\R^n$.
    \end{definition}
    Strictly, Clarke's generalised gradient is a subset of the dual space.
    Throughout this paper we identify differential forms over $\R^n$
    with vectors over $\R^n$ via the isomorphism given by the standard
    Euclidean metric.

    Properties and applications
    of the generalised gradient can be found in \cite{clarke1990optimization}.
    In Theorem \ref{thm:clarke} for any $A\subset\R^n$ we
    denote the closure of the convex hull of $A$ by
    $\overline{\operatorname{co}}\{A\}$.

    Let $f:\R^n\to \R$ be a $C^1$ function.
    The gradient of $f$, in the standard Euclidean metric, will
    be denoted $Df$.

    \begin{theorem}[{\cite[Thm 2.5.1]{clarke1990optimization}}]
      \label{thm:clarke}
      Let $Y\subset\R^n$
      $f:Y\to\R$ be a locally Lipschitz function, and let $S$ be any subset of
      Lebesgue measure zero. Then
      \[
      \partial^\circ f(x)=
        \overline{\operatorname{co}}\{
          \lim Df(x_i):\,x_i\to x,\ x_i\not\in S,\  f\ 
          \mbox{differentiable at }x_i\}.
      \]
    \end{theorem}

    \begin{theorem}[{\cite[Theorem 2.2, page 4475]{czarnecki2006approximation}}]
      \label{thm_unreasonable_approx_result}
      Let $B\subset\R^n$ be the open unit ball with respect to the
      standard Euclidean metric.

      Let $Y\subset\R^n$ be an open subset and let
      $f:Y\to\R$ be a locally Lipschitz function.
      Then, for every continuous function $\epsilon:Y\to\R^+$,
      there exists a smooth function $f_\epsilon:Y\to\R$
      such that for all $x\in Y$,
      \begin{align*}
        \abs{f_\epsilon(x) - f(x)}&\leq \epsilon(x)
      \end{align*}
      and
      \begin{align*}
        Df_\epsilon(x) &\in \partial^\circ f|_{(x + \epsilon(x)B)\cap Y}
          + \epsilon(x) B,\qquad
        \partial^\circ f(x) \subset Df|_{(x + \epsilon(x)B)\cap Y}
          + \epsilon(x) B.
      \end{align*}
    \end{theorem}

  \subsection{Limit curve theorems for complete spaces}

    We will make use of standard limit curve theorems 
    \cite[Page 369]{galloway1986curvature} and the upper semi-continuity 
    of the Lorentzian 
    length function $L$ with respect to the topology of uniform convergence on
    compact sets \cite[Proposition, page 369]{galloway1986curvature}.
    To the best of our knowledge these results were first presented in
    \cite{galloway1986curvature}, though Galloway does indicate that
    the first edition of \cite{beem1996global} contains similar results.

    The following is a typical version of ``the'' limit curve theorem.
    For a modern and more flexible 
    statement of the result see \cite{minguzzi2008limit}.
    \begin{theorem}[{A paraphrase of \cite[Lemma 14.2]{beem1996global}}]
      \label{thm_trad_lim_curve}
      Let $\gamma_i:\R^+\to M$ be a sequence of future directed 
      causal curves parametrised
      with respect to the arc-length induced by a complete Riemannian metric.
      If $x$ is an accumulation point of $(\gamma_i(0))_i$ then
      there exists an inextendible future directed causal curve
      $\gamma:\R^+\to M$ so that $\gamma(0)=x$ and a subsequence
      $(\gamma_{i_k})$ which converges to $\gamma$ uniformly 
      on compact subsets of $\R^+$ with respect to the distance induced
      by the Riemannian metric.
    \end{theorem}

    The upper semi-continuity of the Lorentzian length functional
    $L$ has been expressed in a handful of
    different ways, for example
    \cite[Proposition, page 369]{galloway1986curvature},
    \cite[Theorem 7.5]{penrose1972techniques}, 
    \cite[Proposition 14.3]{beem1996global}, and 
    \cite[Theorem 2.4]{minguzzi2008limit}.
    We rephrase Geroch's formulation
    \cite[Proposition, page 369]{galloway1986curvature}.
    
    \begin{theorem}
      \label{thm_upper_semi_cts}
      Let $I\subset\R$ be a compact interval
      and let $d$ be a metric on $M$ which is compatible with the
      manifold topology.
      For each $i\in\N$, let $\gamma_i:I\to M$ be a past directed 
      inextendible causal curve. 
      If the sequence $(\gamma_i)$ converges uniformly, with respect
      to a metric $d$,
      to a past directed inextendible causal curve $\gamma:I\to M$,
      then 
      \[
        L(\gamma)\geq\limsup_iL(\gamma_i).
      \]
    \end{theorem}
    A proof can be found in
    \cite[Proposition 14.3]{beem1996global}.
    Theorem \ref{thm_upper_semi_cts}
    does not necessarily hold if the curves are defined over a non-compact interval.
    Examples \ref{ex_convergence_in_incomplete} and
    \ref{ex_convergence_in_incomplete_complete_riemanna} below
    demonstrate this.
        
    In limit curve 
    theorems for Lorentzian manifolds, the
    limit curve is causal.
    This has been proven
    elsewhere and is well-known,
    see for instance \cite[Lemma 2.7]{minguzzi2008limit} or
    \cite[Second paragraph page 77]{beem1996global}.
    For completeness we include the following result, using a
    proof that we have not seen elsewhere which is inspired
    by \cite[Lemma 3.29]{beem1996global}.

    \begin{lemma}
      \label{lem.pointwiseconvergenceforcausalcharacter}
      Let $(M,g)$ be a Lorentzian manifold.
      For each $i\in\N$,
      let $\gamma_i:(a,b)\to M$ be a future directed causal curve.
      If there exists a continuous curve
      $\gamma:(a,b)\to M$ so that
      for all $t\in (a,b)$, $\gamma_i(t)\to\gamma(a)$ then
      $\gamma$ is a continuous causal curve.
    \end{lemma}
    \begin{proof}
      Let $t\in(a,b)$ and choose $U$ an open convex normal neighbourhood
      containing $\gamma(t)$. Since $U$ is open there exists
      $\epsilon>0$ so that $\gamma(t-\epsilon,t+\epsilon)\subset U$.
      Let $t_1,t_2\in(t-\epsilon,t+\epsilon)$, $t_1<t_2$.
      By restricting to a subsequence, if necessary we can assume that
      for all $i\in\N$, $\gamma_i(t_1),\gamma_i(t_2)\in U$.
   
      Since each $\gamma_i$ is future directed causal
      and as $U$ is convex normal, for each $i\in\N$ there exists a future directed
      causal geodesic in $U$ from $\gamma_i(t_1)$ to $\gamma_i(t_2)$.
      Let $v_i=\exp^{-1}_{\gamma_i(t_1)}(\gamma_i(t_2))$.
      Then by the assumption of pointwise convergence of $\gamma_i$ and joint
      continuity of $\exp^{-1}$, the sequence of tangent vectors
      $(v_i)_{i\in\N}$ converges to $v=\exp^{-1}_{\gamma(t_1)}(\gamma(t_2))$.

      Let the vector field $T$ define our time orientation.
      Since each $\gamma_i$ is future directed and timelike we have
      $g(T, v_i)\leq 0$ and $g(v_i,v_i)\leq 0$.
      Taking the limit with respect to $i$ shows that
      $g(T,v)\leq 0$ and $g(v,v)\leq 0$.
      Hence $v\in T_{\gamma(t_1)}M$ is future directed and causal.
      By construction $\exp_{\gamma(t_1)}(v)=\gamma(t_2)$.
      Thus as $U$ is convex normal, the unique geodesic between
      $\gamma(t_1)$ and $\gamma(t_2)$ is the curve
      $t\mapsto\exp_{\gamma(t_1)}(tv)$, which is future directed.
      Thus $\gamma(t_2)\geq \gamma(t_1)$ and so $\gamma$ is a continuous causal
      curve as required.
    \end{proof}

  \subsection{Lorentzian length control}
    \label{sec_llc}

    In this section we prove one result,
    Lemma
    \ref{lem_improved_length_bound},
    which gives sufficient conditions
    to know when control of Lorentzian length over compact subsets of the
    domain of the curves implies control of Lorentzian length over the entire
    domain of the curves.
    Application of
    Lemma
    \ref{lem_improved_length_bound}
    depends, in this paper, on a careful analysis of
    parametrisations of causal curves in Lorentzian manifolds
    and other additional strong hypotheses: see Section 
    \ref{sec_length_control}.
    After proving Lemma \ref{lem_improved_length_bound},
    we give our definition of convergence of curves in incomplete 
    metric spaces, Definition \ref{defn:limit-curve}, and then  present
    two examples illustrating what can go wrong.

    \begin{lemma}
      \label{lem_improved_length_bound}
      Let $(M,g)$ be a Lorentzian manifold, and fix $a\in\R^+$, finite.
      Let $d$ be a metric on $M$ which is compatible with the manifold
      topology.
      Let $(\gamma_i:[0,a)\to M)_i$ be a sequence of 
      causal curves which converge uniformly on compact subsets of
      $[0,a)$, with respect to $d$, to a causal curve $\gamma:[0,a)\to M$.
      If
      \begin{enumerate}[noitemsep]
        \item the Lorentzian arc-lengths $L(\gamma_i)$
          are all finite, the sequence $(L(\gamma_i))_i$ is bounded above,
          and
        \item there exists $b>0$ so that
          for all $i\in\N$ and all 
          $t_1,t_2\in[0,a)$
          \[
            L(\gamma_i|_{[t_1,t_2]})\leq b\abs{t_2 - t_1}
            \quad\mbox{and}
            \quad L(\gamma_i|_{[t_1,a)})\leq b\abs{a - t_1},
          \]
      \end{enumerate}
      then there exists a subsequence
      $(\gamma_{i_k})_k$ so that $L(\gamma)\geq\limsup_kL(\gamma_{i_k})$.
      Moreover, if for all $t\in[0,a)$,
      \[
        L(\gamma|_{[0,t]})=\lim_{k} L(\gamma_{i_k}|_{[0,t]})
      \]
      then $L(\gamma)=\lim_kL(\gamma_{i_k})$.
    \end{lemma}
    \begin{proof}
      We begin by noting that $(L(\gamma_i))_i$ is a bounded sequence
      in $\R^+$. We can therefore choose
      a subsequence $(L(\gamma_{i_k}))_k$ so that
      $\lim_kL(\gamma_{i_k}) = \limsup_i L(\gamma_{i})$.

      We can define a function
      $f:[0,a)\times\N\to\R^+$ by
      $f(t,k)=L(\gamma_{i_k}|_{[0,t]})$. 
      By definition of the Lorentzian length,
      Equation \eqref{eq_lg_def},
      $L(\gamma_{i_k}|_{[0,t]})\to L(\gamma_{i_k})$
      as $t\to a$. Therefore
      we can extend $f$ to the domain $[0,a]\times\N$
      by defining $f(a,k)=L(\gamma_{i_k})$ and ensure that, for
      fixed $k$, $f(\cdot, k)$ is continuous.

      Choose $\epsilon>0$. 
      We need to check three cases. 
      For all
      $t,s\in[0,a)$, and all $k\in\N$
      if $0<t-s<\epsilon/b$ then
      \[
        {f(t,k) - f(s,k)} 
        = {L(\gamma_{i_k}|_{[0,t]}) - L(\gamma_{i_k}|_{[0,s]})}
        = {L(\gamma_{i_k}|_{[s,t]})}
        \leq {b(t-s)}<\epsilon.
      \]
      A similar calculation can be performed if $s>t$ to show that
      if $\abs{t-s}<\epsilon/b$ then
      $\abs{f(t,k) - f(s,k)} \leq \epsilon$.
      Further, for all $t\in[0,a]$ and all $k\in\N$
      if $0<a-t<\epsilon / b$ then
      \[
        {f(a,k) - f(t,k)} 
        = {L(\gamma_{i_k}|_{[0,a)}) - L(\gamma_{i_k}|_{[0,t]})}
        = {L(\gamma_{i_k}|_{[t,a)})}
        \leq {b(a-t)}<\epsilon.
      \]
      That is, for all $t,s\in[0,a]$ and all $k\in\N$,
      $\abs{t-s}<\epsilon/ b$ implies that
      $\abs{f(t,k) - f(s,k)}<\epsilon$.

      Thus the family of functions $f(\cdot, k)$ is uniformly bounded
      and equicontinuous. 
      Arzel\'a-Ascoli, \cite[Theorem XII.6.4]{dugundij1966top},
      therefore implies that there exists 
      a function $f:[0,a]\to \R^+$ and a subsequence
      $(f(\cdot, k_j))_{j\in\N}$ 
      which converges to $f$ uniformly on compact subsets of $[0,a]$,
      and hence on $[0,a]$ (this was the point of assuming $a<\infty$). 
      The upper semicontinuity of the Lorentzian length function $L$
      and Theorem \ref{thm_upper_semi_cts} tell us that 
      \[
        L(\gamma) = \lim_{t\to a}L(\gamma|_{[0,t]})
        \geq \lim_{t\to a}\limsup_{j} L(\gamma_{i_{k_j}}|_{[0,t]})
        =\lim_{t\to a}\limsup_{j}  f(t, k_j)=\lim_{t\to a}\lim_{j}  f(t, k_j).
      \]

      Since $f(\cdot, k)\to f(\cdot)$ uniformly on $[0,a]$ we can apply the
      Moore-Osgood theorem, \cite[Theorem 7.11]{rudin1976principles}
      or \cite[Theorem VII.2, page 100]{graves2012theory}.
      That is, we can interchange the limits and compute that
      \begin{align}
        L(\gamma) = \lim_{t\to a}L(\gamma|_{[0,t]})
          &\geq \lim_{t\to a}\limsup_{j} L(\gamma_{i_{k_j}}|_{[0,t]})\nonumber\\
          &= \lim_{t\to a}\lim_{j} f(t, k_j)\nonumber\\
          &= \lim_{j} \lim_{t\to a}f(t, k_j)\nonumber\\
          &= \lim_{j} L(\gamma_{i_{k_j}})
           = \lim_{k} L(\gamma_{i_k})
           = \limsup_{i} L(\gamma_{i}).
           \label{eq:two-ineqs}
      \end{align}
      If $L(\gamma|_{[0,t]})=\lim_kL(\gamma_{i_k}|_{[0,t]})$ then
      it is clear that the inequality in Equation \eqref{eq:two-ineqs} is
      an equality.
      Thus the result holds.
    \end{proof}

    Lemma \ref{lem_improved_length_bound} gives conditions under which the 
    Lorentzian length function is upper semi-continuous on non-compact inextendible
    causal curves.
    To use a result like
    Lemma \ref{lem_improved_length_bound}, one would typically select
    a suitable sequence of causal curves and then control the length of the limit
    curve given by Theorem \ref{thm_trad_lim_curve}.

    Theorem \ref{thm_trad_lim_curve}
    requires the curves in 
    the sequence to be $1$-Lipschitz parametrised with respect to the
    chosen complete distance on $M$. 
    Hence to
    achieve the conditions of
    Lemma 
    \ref{lem_improved_length_bound}, the chosen distance
    must be tied to the
    Lorentzian geometry: the null distance
    \cite{sormani2016null} of a suitable time function provides such
    a distance. That said, the null distance of a suitable time function has no
    guarantee to be complete. 

    In the remainder of this paper we 
    prove a generalisation of
    Theorem \ref{thm_trad_lim_curve} for distances that are not complete
    and then show that the null distance associated to various time functions
    allows us to use 
    Lemma \ref{lem_improved_length_bound}.

  \subsection{What goes wrong in incomplete spaces}
    \label{sec_what_goes_wrong}
    
    In this section we present two examples that demonstrate how,
    for a sequence $(\gamma_i)_i$, of
    causal curves, that uniformly converge
    to a causal curve $\gamma:\R^+\to M$ it is possible for
    \[
      L(\gamma)<\limsup_iL(\gamma_i)
    \]
    despite Theorem \ref{thm_upper_semi_cts} which
    shows that on compact subsets $K\subset \R^+$ we have
    \[
      L(\gamma|_K)\geq\limsup_iL(\gamma_i|_K).
    \]

    Ultimately the problem stems from needing to choose an auxiliary metric on
    the manifold to define convergence, as the Lorentzian distance is not
    suitable for this purpose. Having chosen a distance in order to define
    uniform convergence, the parametrisation induced
    by the distance is typically not at all
    related to the Lorentzian 
    length functional. For example, the usual method is to choose a complete
    Riemannian metric. In order to ensure an appropriate relation between the
    distance and the Lorentzian length, we typically lose completeness, and so
    it is
    necessary to allow the metric space $(M,d)$ to be incomplete.
    This causes a secondary problem: what does it mean for curves to converge in an
    incomplete metric space?
    
    \begin{figure}[h]
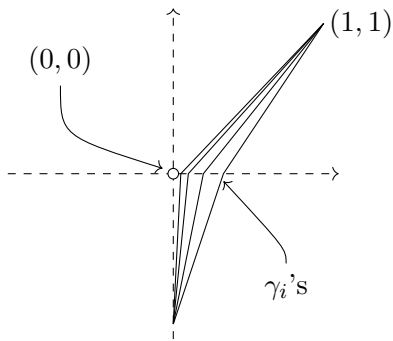

      \centering
      \tikz{
        \draw[dashed] (0,-2.2)--(0, -2pt);
        \draw[dashed, ->] (0,2pt)--(0, 2.2);
        \draw[dashed] (-2.2, 0)--(-2pt, 0);
        \draw[dashed, ->] (2pt, 0)--(2.2, 0);

        \draw (0,0) circle (2pt);

        \foreach \i in {3, 5, 10, 20}
          \draw (2,2) -- (2/\i, 0) -- (0,-2);

        \node(O) at (-1.5, 1.5) {$(0,0)$};
        \draw[->] (O) .. controls (-1.5,0.5) .. (-4pt,2pt);
        \node(G) at (1.5, -1.5) {$\gamma_i$'s};
        \draw[->] (G) .. controls (1.5, -1) .. (2/3, -0.1);
        \node at (2.5,2) {$(1,1)$}
      }
      \caption{An illustration of Example \ref{ex_convergence_in_incomplete}.
        The circle represents the removed origin in $\R^2$. The dashed lines
        represent the $x$ and $y$ axis. The remaining curves
        present the images of three of the $\gamma_i$'s. The
        intersection of these curves with the $x$ axis tends to the
        removed origin as $i\to\infty$.}
      \label{fig_minkow_hole}
    \end{figure}

    Figure \ref{fig_minkow_hole} illustrates the problems we face when defining
    what we mean by a limit curve in an incomplete space. 
    The figure illustrates a sequence of past directed causal curves
    $\gamma_i:[0,2]\to \R^2\setminus\{(0,0)\}$ whose images are two straight lines 
    between the points $(1,1)$, $(1/i, 0)$ and $(0,1)$.
    Whilst it is
    intuitively clear that the sequence of curves converges pointwise to a
    ``disconnected curve'', 
    \[
      t\mapsto \left\{\begin{aligned}
        &(1,1) - t(1,1),&& t\in [0,1),\\
        &(0,0) - (t-1)(0,1), && t\in (1,2],
      \end{aligned}\right.
    \]
    this will not do as a basis for the definition of
    limit curve.

    Our definition of curve requires a 
    connected domain (and so image). As well as
    satisfying an intuitive notion of curve, this requirement is essential for
    discussion of causality: {\em causal curves must be traversable}.

    Thus if we take the curves in Figure \ref{fig_minkow_hole} to start at the
    top right, the point $(1,1)$, and be past directed, 
    the best we can hope for is that the ``top
    half'' of the curves converge to the ``top half'' of the intuitive limit
    curve. 
    The second halves of the limit curves
    will not be said to be converging to a limit curve.
    They ``fall off'' the limit.
    For us to talk about
    a limit curve requires specification of a point which is the start of
    the limit curve.
    In this case the point is $(1,1)$.
    We present the details in Example \ref{ex_convergence_in_incomplete}.

    \begin{definition}
      \label{defn:limit-curve}
      Let ${M}$ be a manifold, and
      $d:{M}\times{M}\to\mathbb{R^+}$ a metric on $M$.
      Let  $\gamma_i:[0,a_i)\to{M}$, $a_i\in\R^+\cup\{\infty\}$, be a sequence of
      $C^0$ curves in ${M}$ with $\gamma_i(0)\to x\in M$.
      We say that $(\gamma_i)$ $d$-converges uniformly on compacta to
      the continuous curve
      $\gamma:[0,a)\to M$ if $\gamma(0)=x$ and for all $\delta<a$, 
      there exists $N_\delta\in\N$ so that
      $i\geq N_\delta$ implies that
      $a_i>\delta$ and
      $(\gamma_i|_{[0,\delta]})_{i\geq N_\delta}$ converges 
      uniformly to $\gamma|_{[0,\delta]}:[0,\delta]\to M$
      with respect to the distance $d$.
    \end{definition}

    If $\gamma:[0,a)\to M$ is a limit curve then for
    all $t\in[0,a)$ we have that
    $\gamma|_{[0,t)}:[0,t)\to M$ is also a limit curve.
    That is, sub-curves of a limit curve are also limit curves.
    In this sense limit curves are non-unique. Uniqueness can be
    achieved 
    by taking the maximal limit curve, whose graph is the union of
    the graphs of all possible limit curves. In practice, we will
    work with the maximal curve.

    \begin{example}
      \label{ex_convergence_in_incomplete}
      Let $M=\R^2\setminus\{(0,0)\}$ have the
      metric $dx^2-dy^2$, where $x,y$ are the standard
      coordinates on $M$ induced by inclusion into $\R^2$.
      For $i\geq 2$, let $\gamma_i:[0,2]\to M$ be defined by
      \[
        \gamma_i(t)=\left\{\begin{aligned}
          &(1, 1) - t(1-1/i, 1), && t\in [0,1)\\
          &(1/i, 0) - (t-1)(1/i, 1), && t\in [1,2].
        \end{aligned}\right\}
      \]
      See Figure \ref{fig_minkow_hole} for a diagram representing 
      $M$ and the $\gamma_i$'s.
      Each $\gamma_i$ is a $C^0$ piecewise smooth causal curve, and
      direct calculation shows that the length of $\gamma_i$ defined by the
      Lorentzian metric is
      $L(\gamma_i)=\sqrt{1 - 1/i^2} + \sqrt{1-(1-1/i)^2 }$.
      Thus $L(\gamma_i)\to 1$.
      According to Definition \ref{defn:limit-curve}, 
      the maximal limit curve from the point $(1,1)$
      is $\gamma:[0,1)\to M$ defined by
      $\gamma(t)=(1, 1) - t(1,1)$.
      Direct calculation gives $L(\gamma)=0$. 
      
      Thus, in particular, $L(\gamma)<\limsup_i L(\gamma_i)$.
      Since the Lorentzian arc length is invariant to changes in parametrisation
      no amount of fiddling with parametrisation or the distance used to 
      to define the limit curve will result in $L(\gamma)=\limsup_i L(\gamma_i)$.
      The situation described in this example can be excluded
      by assuming that $M$ is globally hyperbolic, since the problems stem from
      the set
      $\jfuture{(0,-1)}\cap\jpast{(1,1)}$ not being compact.
      \qedhere
    \end{example}

    \begin{example}
      \label{ex_convergence_in_incomplete_complete_riemanna}
      Continuing Example \ref{ex_convergence_in_incomplete},
      we show what we need to do to the $\gamma_i$'s to apply
      Theorem \ref{thm_trad_lim_curve} and how this relates to
      the conditions of
      Lemma \ref{lem_improved_length_bound}.
      Let $h=\frac{1}{(x^2 + y^2)^2}(dx^2 + dy^2)$. This is a complete Riemannian
      metric on $M$.

      Classical limit curve theorems use a distance induced
      by a complete Riemannian metric such as $h$, and assume that
      the curves they apply to are inextendible. Doing so
      ensures that the inextendible curves being considered will have domains
      $\R^+$ when re-parametrised to be $1$-Lipschitz
      with respect to $d(\cdot, \cdot; h)$.
      To compare with classical limit theorem statements, we now construct
      inextendible extensions
      of the
      $\gamma_i$. We construct the extensions by adding a null curve, showing that
      the problem with the limits of Lorentzian lengths does not arise from the
      extension process.

      Let $\lambda:[2,\infty)\to M$ be given by
      $\lambda(t)=(0, -1) - (t-2)(1,1)$.
      We now extend each $\gamma_i$ to be a past directed, inextendible
      curve by appending $\lambda$. That is, our extension is the curve
      given by
      \[
        t\mapsto \left\{\begin{aligned}
          &\gamma_i(t), && t\in [0,2),\\
          &\lambda(t), && t\in[2,\infty).
        \end{aligned}\right.
      \]
      We shall denoted these extensions by 
      $\gamma_i:\R^+\to M$.
      The limit curve $\gamma$ described in Example
      \ref{ex_convergence_in_incomplete} is already inextendible.

      We now compute the $h$ arc-length parametrisations
      of our inextendible $\gamma_i$'s and $\gamma$.
      Let $s:[0,1)\to\R^+$ be defined by
      \[
        s(t) = L(\gamma|_{[0,t]};h)=\int_0^t\frac{1}{\sqrt{2}(1-\tau)^2}\dd\, \tau
        =\frac{1}{\sqrt{2}}\left(\frac{t}{1-t}\right).
      \]
      Note that $s(t)\to\infty$ as $t\to 1$.
      We can define $t:\R^+\to[0,1)$ by, $s(t(u))=u$.
      That is,
      $
        t(s) = \frac{\sqrt{2}s}{1 + \sqrt{2}s}
      $
      so that $\gamma\circ t:\R^+\to M$ is an $h$ arc-length parametrised
      curve.

      The curves $\gamma_i:\R^+\to M$ are inextendible and therefore
      $L(\gamma_i;h)=\infty$.
      Thus, for each $i\in\N$ define
      a function $s_i:\R^+\to \R^+$,  
      by $s_i(t)=L(\gamma_i|_{[0,t)};h)$.
      That is, 
      \begin{equation*}
        \label{eq_harclengthone}
        s_i(t)
        =
          \scaleint{40pt}_{\!\!\!\!\!\!0}^t{%
          \frac{\sqrt{\left(1 - \frac{1}{i}\right)^2 + 1}}%
          {\left(1-\tau\left(1-\frac{1}{i}\right)\right)^2 
            + (1-\tau)^2}\,
            \dd\,\tau},
      \end{equation*}
      for $t\in [0,1)$, and 
      \begin{equation*}
        \label{eq_harclengthtwo}
        s_i(t) = \lim_{t\to 1}s_i(t)
          +
          \scaleint{40pt}_{\!\!\!\!\!\!1}^{t}{%
          \frac{\sqrt{\left(\frac{1}{i}\right)^2 + 1}}%
          {\left(\frac{2-\tau}{i}\right)^2 
          + (\tau-1)^2}\,
            \dd\,\tau}
      \end{equation*}
      for $t\in [1,2)$
      and
      \begin{equation*}
        \label{eq_harclengththree}
        s_i(t) = \lim_{t\to 2}s_i(t)
          +
          \scaleint{40pt}_{\!\!\!\!\!\!2}^{t}{%
          \frac{\sqrt{2}}%
          {\left(\tau^2 \right)^2 
          + (\tau - 2)^2}\,
            \dd\,\tau}
      \end{equation*}
      for $t\in[2,\infty)$. 
      We can now define 
      $t_i:\R^+\to\R^+$ by  $s_i(t_i(u))=u$.

      All curves $\gamma\circ t$ and $\gamma_i\circ t_i$ now have the
      same domain $\R^+$ and are such that for all $u_1,u_2\in\R^+$
      $d(\gamma\circ t(u_1),\gamma\circ t(u_2); h)\leq\abs{u_1-u_2}$ and
      $d(\gamma_i\circ t_i(u_1),\gamma_i\circ t_i(u_2); h)\leq\abs{u_1-u_2}$.

      Since 
      \[
        L(\gamma\circ t)=0<1=\limsup_iL(\gamma_i\circ t_i)
      \]
      we know that Lemma
      \ref{lem_improved_length_bound} does not hold for the
      sequence $(\gamma_i\circ t_i)$ and the limit curve $\gamma\circ t$.
      We now investigate the conditions of Lemma
      \ref{lem_improved_length_bound} to determine which fails.
      Since all curves, $\gamma$ and $\gamma_i$, have domain $\R^+$
      and 
      Lemma
      \ref{lem_improved_length_bound} requires our curves to have a domain
      with compact closure in $\R^+$ we shall first need to re-parametrise.
      For each $i\in\N$ we consider $\gamma_i\circ t_i\circ \tan:[0,\pi/2)\to M$
      and $\gamma\circ t\circ \tan:[0,\pi/2)\to M$.
      We immediately see that 
      the sequence $(\gamma_i\circ t_i\circ \tan)_i$ converges
      uniformly to 
      $\gamma\circ t\circ \tan$.
      It is also clear that
      $L(\gamma_i\circ t_i\circ\tan)$
      is finite for all $i\in\N$ and the sequence $(L(\gamma_i\circ
      t_i\circ\tan))_i$
      is bounded above. 

      The sequence $(s_i(1))_i$ is increasing with limit $\infty$.
      In particular, $(\arctan\circ s_i(1))_i$ is increasing with
      limit $\pi/2$.
      This implies that
      \[
        1=L(\gamma_i\circ t_i\circ\tan|_{[\arctan(s_i(1)),\pi/2)})
        \leq b \left(\frac{\pi}{2}-\arctan(s_i(1))\right).
      \]
      While the left hand side is constant the right hand side
      tends to $0$ as $i\to\infty$.
      This violates the second listed assumption of
      Lemma \ref{lem_improved_length_bound}.
    \end{example}
  
\section{The limit curve theorem for incomplete metric spaces}
  \label{sec_general_limit_curve}

  In this section we generalise Theorem \ref{thm_trad_lim_curve} to
  incomplete metric spaces. 
  We begin by considering the limit curve theorem
  in a compact subset of $M$. First we prove a lemma that will help us
  re-parametrise our curves.

  \begin{lemma}
    \label{lem_uni_for_repara}
    For each $i\in\N$ let $a_i\in\R^+\cup\{\infty\}$
    and define 
    \[
      Y_i=\left\{\begin{aligned}
        &[0,a_i], && a_i<\infty,\\
        &[0,a_i), && a_i=\infty.
      \end{aligned}\right.
    \]
    Let
    $a=\limsup_ia_i$ and define
    \[
      X=\left\{\begin{aligned}
        &[0,a], && a<\infty,\\
        &[0,a), && a=\infty.
      \end{aligned}\right.
    \]
    For each $i\in\N$ define
    $f_i:X\to Y_i$ a bijective, continuous, increasing
    function by
    \[
      f_i(x)=\left\{
      \begin{aligned}
        &x, && a_i=a=\infty,\\
        &\frac{2a_i}{\pi}\arctan\left(\frac{\pi}{2a_i}x\right),
          && a=\infty,\, a_i<\infty,\\
        &\frac{a_i}{a}x, && a_i,\, a<\infty.
      \end{aligned}
      \right.
    \]
    For all $\epsilon>0$ there exists a subsequence
    $(a_{i_k})_k$ of $(a_i)_i$ such that for all
    $k\in\N$, $a_{i_k}\leq a+\epsilon$ and so
    $\sup\{f_{i_k}'(t):t\in X\}\leq 1+\epsilon$.
  \end{lemma}
  \begin{proof} Let $\epsilon>0$.
    Since $a=\limsup_ia_i$ there exists a subsequence
    $(a_{i_k})_k$ so that $a_{i_k}\to a$. Consequently,
    there exists $N$ big enough so that the subsequence
    $(a_{i_{k+N}})_{k+N}$ is such that for all $k\in\N$,
    $a_{i_{k+N}}\leq a + \epsilon$. 

    Passing to such a subsequence if necessary, for sufficiently large $i$ we have
    \[
      0<f_i'(x)=\left\{
      \begin{aligned}
        &1, && a_i=a=\infty,\\
        &\frac{1}{1+\left(\frac{\pi}{2a_i}x\right)^2},
          && a=\infty,\, a_i<\infty,\\
        &\frac{a_i}{a}, && a_i,\, a<\infty.
      \end{aligned}
      \right.
      \leq
      \left\{
      \begin{aligned}
        &1, && a_i=a=\infty,\\
        &1, && a=\infty,\, a_i<\infty,\\
        &1 + \epsilon, && a_i,\, a<\infty.
      \end{aligned}
      \right.
    \]
    Thus the result holds.
  \end{proof}

  The following lemma is a ``local'' version of Theorem
  \ref{thm_trad_lim_curve}. In Theorem 
  \ref{thm_CurveUniformConvergenceInBoundedRegion_extension},
  below, we show how to use Proposition
  \ref{prop_CurveUniformConvergenceInBoundedRegion}
  to prove a global version of the limit curve theorem
  for incomplete distances. 
  In this local version, 
  we manage incompleteness by
  working in a compact set.

  \begin{proposition}
    \label{prop_CurveUniformConvergenceInBoundedRegion}
    Let ${M}$ be a manifold, let
    $d:{M}\times{M}\to\mathbb{R}$ be a distance on $M$ compatible with
    the manifold topology, and
    let $B$ be a compact subset of $M$. 
    Let $\gamma_i:Y_i=[0,a_i)\to{M}$, $a_i\in\R^+\cup\{\infty\}$, be a sequence of
    $C^0$ curves in ${M}$ so that for some $N\in\N$,
    $n>N$ implies that $\gamma_n\subset B$.
    Let the functions
    $f_i:X\to Y_i$ be constructed as in Lemma
    \ref{lem_uni_for_repara}. 
    If for all $i\in\N$ and for all $t_1,t_2\in I_i$,
    \[
      d(\gamma_i(t_1),\gamma_i(t_2))\leq |t_1-t_2|.
    \]
    then there is a subsequence $(\gamma_{i_k}\circ f_{i_k})_k$ of
    $(\gamma_i\circ f_i)_{i}$ which
    \begin{enumerate}
      \item converges uniformly 
        on compact subsets of $X$ to a $C^0$ curve
        $\gamma:X\to B$ with respect to $d$, and
      \item is such that $\lim_ka_{i_k}=\limsup_{i}{a_i}$. 
    \end{enumerate}
  \end{proposition}
  \begin{proof}
    Let $\epsilon>0$.
    Since $a=\limsup\{a_i:i\in\N\}$, for all sufficiently large
    $i\in\N$ we have $a_i< a+\epsilon$.
    Thus there exists
    a subsequence $(a_{i_k})_k$ so that $\lim_ka_{i_k}=a$
    and for all $k\in\N$, $a_{i_k}< a + \epsilon$.
    Without loss of generality we assume that $i_k=k$.
    That is, we assume that, $\lim_ia_i=a$ and
    for all $i\in\N$ $a_i<a+\epsilon$.
    Lemma \ref{lem_uni_for_repara} shows that 
    $\sup\{f_i'(t):t\in X\}\leq 1 + \epsilon$. 
    
    By assumption for all $i\in\N$ and all $t_1,t_2\in I_i$ we know that  
    $d(\gamma_i(t_1),\gamma_i(t_2))\leq|t_1-t_2|$. 
    Thus for all $t_1,t_2\in X$ we know that 
    \begin{align*}
        d(\gamma_i\circ f_i(t_1),\gamma_i\circ f_i(t_2))
        \leq \abs{f_i(t_1)-f_i(t_2)}
        \leq (1+\epsilon)\abs{t_1-t_2}.
    \end{align*}
    Thus
    $\{\gamma_i\circ f_i:i\in\N\}$ is
    uniformly equicontinuous.
    By assumption, for all sufficiently large $i\in\N$ we have
    $\gamma_i\subset B$.
    Thus
    $\{\gamma_i\circ f_i(t):i\in\N\}$ has 
    compact closure for each $t\in X$. 

    Arzel\`a and Ascoli's theorem \cite[Theorem XII.7.6.4 and
    Theorem XII.7.7.2]{dugundij1966top}, 
    implies that 
    there exists some $C^0$ curve $\gamma:X\to{M}$ such
    that there is a subsequence of $({\gamma}_i\circ f_i)_{i\in\N}$ which
    converges uniformly to $\gamma$ on compact subsets of $X$
    with respect to $d$.  
    Since, for $i\in\N$ large enough we know that $\gamma_i\subset B$,
    and as $\gamma(t)$ is the pointwise limit of some subset of
    $(\gamma_i\circ f_i(t))_i$, we see that $\gamma\subset B$.
    Thus the result holds.
  \end{proof}
  The first part of Proposition \ref{prop_CurveUniformConvergenceInBoundedRegion}
  is not a new result: similar ideas are used in 
  \cite[Theorem 3.7]{kunzinger2018lorentzian}. For Lorentzian length control we need
  to understand the relationship between $\limsup_ia_i$, $\lim_ka_{i_k}$
  and $a$. This part of 
  Proposition \ref{prop_CurveUniformConvergenceInBoundedRegion} is new.

  We can patch the limit curves on intersecting compact subsets together.
  When we do so, the ``second half'' of the limit curve may
  ``fall off'' as in Example \ref{ex_convergence_in_incomplete}.
  This occurs as the patching
  mechanism restricts us to working on the traversable portion of the
  limit curve from its assumed starting point.

  \begin{theorem}[The limit curve theorem]
    \label{thm_CurveUniformConvergenceInBoundedRegion_extension}
    Let ${M}$ be a manifold and let
    $d:{M}\times{M}\to\mathbb{R}$ be a distance on $M$ whose topology agrees with
    that of $M$.
    Let $(\gamma_i:[0,a_i)\to{M})_{i\in\N}$,
    $a_i\in\R^+\cup\{\infty\}$, be a sequence of
    $C^0$ curves in ${M}$.
    If there exists $x\in M$ so that
    \[
      \gamma_i(0)\to x\in M
    \]
    and if, for all $i\in \N$ and for all $t_1,t_2\in [0,a_i)$,
    \[
      d(\gamma_i(t_1),\gamma_i(t_2))\leq \abs{t_1-t_2},
    \]
    then there exists a $C^0$ curve 
    $\gamma:[0,a)\to M$, $a\in(0,\limsup_ia_i]$,
    with $\gamma(0)=x$, 
    and a subsequence $(\gamma_{i_k})_k$ of 
    $(\gamma_i)_i$ so that 
    \begin{enumerate}
      \item $a\leq \lim\sup_ka_{i_k}$, and
      \item for all compact $C\subset [0,a)$ there
        exists $J\in\N$ so that $k\geq J$ implies that
        $C\subset [0,a_{i_k})$
        and
        $(\gamma_{i_k}|_C)_{k\geq J}$ converges uniformly
        to $\gamma|_C$ with respect to $d$. That is, $(\gamma_{i_k})_k$
        $d$-converges uniformly to $\gamma$ on compacta, as in Definition
        \ref{defn:limit-curve}.
    \end{enumerate}
  \end{theorem}
    \begin{proof}
    Let $\{K_i:i\in\N\}$ be a compact exhaustion of $M$
    with $x\in K_1$.
    If there exists $j\in\N$ and $N\in\N$ so that $i\geq N$
    implies that $\gamma_i\subset K_j$ then the result follows from 
    Proposition \ref{prop_CurveUniformConvergenceInBoundedRegion}.

    Otherwise for all $j\in \N$ and $N\in\N$ there
    exists $i\geq N$ so that $\gamma_i\not\subset K_j$.
    We can, therefore, by taking a subsequence, assume that
    for all $i,j\in\N$, $\gamma_i\cap K_j\neq\EmptySet$ and 
    that $d(K_j,M\setminus K_{j+1})>0$ for each $j$.
    Thus, if $a_i^j\in[0,a_i)$ is such that
    $\gamma_i|_{[0,a_i^j]}$ is the connected component of
    $\gamma_i\cap K_j$ containing $\gamma_i(0)$, then
    we know that each $a_i^j$ is well-defined and for all $j\in\N$,
    \[
      \limsup_{i} a_i^j
      <
      \limsup_{i} a_i^j + d(K_j,M\setminus K_{j+1})
      \leq 
      \limsup_{i}a_i^{j+1}
      \leq
      \limsup_ia_i.
    \]
    In particular, as $K_j$ is compact,
    $\limsup_{i} a_i^j$ is finite, though $\limsup_ia_i$ may be infinite.

    We shall construct $\gamma$ inductively over the compact exhaustion
    using Proposition \ref{prop_CurveUniformConvergenceInBoundedRegion}. 
    Without loss of generality we assume that for all $i\in\N$,
    $\gamma_i(0)\in K_1$.

    Let $a^1=\limsup\{a_i^1:i\in\N\}$. By compactness of $K_1$,
    $a^1<\infty$.
    Apply 
    Proposition \ref{prop_CurveUniformConvergenceInBoundedRegion}
    with $\epsilon_1>0$
    to construct $\gamma^1:[0,a^1]\to K_1$
    a limit curve of a subsequence of 
    $(\gamma_{i}\circ f_{i}^1)_i$
    which converges uniformly to $\gamma^1$,
    where, for each $i\in\N$, $f_i^1:[0,a^1]\to[0,a_i^1]$ is defined by
    $f_i^1(t)=a_i^1 t / a^1$.
    Let $h_1:\N\to\N$ be the function that selects the uniformly
    convergent subsequence.
    That is, the sequence 
    $(\gamma_{h_1(i)}\circ f_{h_1(i)}^1)_i$
    converges uniformly to $\gamma^1$.
    Note that 
    Proposition \ref{prop_CurveUniformConvergenceInBoundedRegion}
    implies that $\lim_i a^1_{h_1(i)}=a^1$.

    We can repeat this construction, starting with the
    sequence of curves
    $(\gamma_{h_1(i)})_i$ and the compact set $K_2$.
    Let $a^2=\limsup\{a^2_{h_1(i)}:i\in\N\}$.
    Apply 
    Proposition \ref{prop_CurveUniformConvergenceInBoundedRegion}
    with $\epsilon_2>0$
    to construct $\gamma^2:[0,a^2]\to K_2$
    a limit curve of a subsequence of 
    $(\gamma_{h_1(i)}\circ f_{h_1(i)}^2)_i$
    which converges uniformly to $\gamma^2$,
    where, for each $i\in\N$, $f_i^2:[0,a^2]\to[0,a_i^2]$ is defined by
    $f_i^2(t)=a_i^2 t / a^2$.
    Let $h_2:\N\to\N$ be the function that selects the uniformly
    convergent subsequence from $(\gamma_{h_1(i)})_i$.
    That is, the sequence 
    $(\gamma_{h_2(i)}\circ f_{h_2(i)}^2)_i$
    converges uniformly to $\gamma^2$
    and is a subsequence of
    $(\gamma_{h_1(i)}\circ f_{h_1(i)}^1)_i$.
    Note again that 
    Proposition \ref{prop_CurveUniformConvergenceInBoundedRegion}
    implies that $\lim_i a^2_{h_2(i)}=a^2$.
    We can repeat the construction for all $j\in\N$.

    With this construction completed we now give the definition of
    the limit curve $\gamma$ and the required subsequence.
    We show that for all $j\in\N$, the curve $\gamma^j$ is a subcurve of 
    $\gamma^{j+1}$.
    By construction
    $(\gamma_{h_j(i)}\circ f_{h_j(i)}^j)_i$
    converges uniformly to $\gamma^j$
    and
    $(\gamma_{h_{j+1}(i)}\circ f_{h_{j+1}(i)}^{j+1})_i$
    converges uniformly to $\gamma^{j+1}$.
    By construction 
    $(\gamma_{h_{j+1}(i)})_i$
    is a subsequence of
    $(\gamma_{h_j(i)})_i$.
    Therefore $(\gamma_{h_{j+1}(i)}\circ f_{h_{j+1}(i)}^j)_i$
    converges to $\gamma^j$ uniformly while
    $(\gamma_{h_{j+1}(i)}\circ f_{h_{j+1}(i)}^{j+1}(i))_i$
    converges to $\gamma^{j+1}$ uniformly.
    By construction
    \[
      \lim_{i\to\infty}\frac{a_{h_{j}(i)}^{j}}{a^{j}}t
      = t
      = \lim_{i\to\infty}
      \frac{a_{h_{j+1}(i)}^{j+1}}{a^{j+1}}t.
    \]
    Thus 
    \[
      \lim_{i\to\infty}\left(f_{h_{j+1}(i)}^{j+1}\right)^{-1}\circ
        f_{h_{j+1}(i)}^{j}(t) = t
    \]
    and hence
    \[
    \left(\gamma_{h_{j+1}(i)}\circ f^{j+1}_{h_{j+1}(i)}\circ
      \left(f_{h_{j+1}(i)}^{j+1}\right)^{-1}\circ
        f_{h_{j+1}(i)}^{j}(t)\right)_i
    \]
    converges to both $\gamma^j(t)$ and $\gamma^{j+1}(t)$.
    Uniform convergence and the Hausdorffness
    of $M$ imply that $\gamma^j(t)=\gamma^{j+1}(t)$. 
    Since this is true for all $t\in [0,a^j]$ we see that $\gamma^j$
    is a subcurve of $\gamma^{j+1}$.

    We choose the subsequence which has the given uniform convergence
    property.
    Let $k\in \N$
    and let $i_k=h_k(k)$.
    Let $j\in\N$ then
    $(\gamma_{i_k}\circ f_{i_k}^j)_{k\geq j}$ is
    a subsequence of 
    $(\gamma_{h_j(i)}\circ f_{h_j(i)}^j)_{i}$ and therefore
    converges uniformly to $\gamma^j$.

    We define the curve $\gamma$.
    As $K_i$ is a compact exhaustion we know that
    for all $j\in\N$, $a^j<a^{j+1} < \limsup_ia_i$.
    Thus the sequence $(a^j)$ is increasing and $\lim_ja^j$ exists (possibly
    infinite). We 
    let $a=\lim_ja^j$.
    Thus, we can define a curve
    $\gamma:[0,a)\to M$
    by $\gamma(t)=\gamma^j(t)$ for any
    $j\in\N$ with $t<a^j$. Since $a^j=\lim_ka^j_{i_k}$ we know that 
    $a^j<\lim\sup_ka_{i_k}$ and therefore that $a\leq \lim\sup_ka_{i_k}$
    (again, possibly infinite).

    We now show that if $C\subset [0,a)$ is compact then 
    there exists $J\in \N$ so that
    $k\geq J$ implies that $\sup\{c:c\in C\}\leq a_{i_k}$
    and
    $(\gamma_{i_k}|_C)_{k\geq J}$ converges uniformly to $\gamma|_C$.

    Let $t=\sup\{c:c\in C\}$.
    Then the image $\gamma([0,t])$ is compact and hence
    lies in some $K_j$, $j\in\N$.
    This implies that $\gamma([0,t])\subset\gamma^j$
    and therefore that $t<a^j$. Since $\lim_k a^j_{i_k}=a^j$
    we know that there exists the required $J$ and that
    the claimed uniform convergence occurs.
  \end{proof}
  Unlike similar results,
  Theorem 
  \ref{thm_CurveUniformConvergenceInBoundedRegion_extension}
  does not require the distance to be complete,
  e.g.\
  \cite[Theorem 14.2, see the last sentence on page 510]{beem1996global},
  \cite[Theorem 3.1]{minguzzi2008limit}
  \cite[Theorem 3.14]{kunzinger2018lorentzian}.

\section{Length control over limit curves via null distances}
  \label{sec_length_control}

  We left open the choice of distance in 
  Theorem \ref{thm_CurveUniformConvergenceInBoundedRegion_extension}.
  In this section we show that 
  the null distance induced by a suitable time function
  is sufficient to show that
  Lemma \ref{lem_improved_length_bound} holds for
  the limit curve constructed in
  Theorem \ref{thm_CurveUniformConvergenceInBoundedRegion_extension}.

  By a ``suitable time function''
  we mean any function whose null distance is a metric compatible with the
  manifold topology as in Proposition \ref{prop:SV-metric} and which satisfies
  a further condition given in 
  Proposition \ref{lem_cauchy_sort_of}. To prepare, the next two results make
  weaker assumptions, and so the null distances my not satisfy Proposition
  \ref{prop:SV-metric}.

	\begin{proposition}
		\label{lem_time_function_parametrisation}
    Let $(M,g)$ be a Lorentzian manifold.
		Let $f:M\to\R$ be a locally Lipschitz function
    which is increasing on all timelike curves, 
    and let $d(\cdot,\cdot; f)$ be
    the associated null distance.
		Let
    $\gamma:[0,a)\to M$, $a\in\R^+\cup\{\infty\}$ be a past directed causal curve.
    If there exists a constant $b>0$ such that for all points where
    $\nabla f$ exists we have $g(\nabla f,\nabla f)\leq -b^2$,
    then for all $t_1,t_2\in[0,a)$
    \[
			d(\gamma(t_1),\gamma(t_2); f)= 
      \abs{f(\gamma(t_2))-f(\gamma(t_1))}
      \geq bL(\gamma|_{[t_1,t_2]}),
    \]
    where $d(\cdot,\cdot; f)$ is the null distance defined by $f$
    and $L$ is the Lorentzian length.
	\end{proposition}
  \begin{proof}
    The equality
    \[
			d(\gamma(t_1),\gamma(t_2); f)= 
      \abs{f(\gamma(t_2))-f(\gamma(t_1))}
    \]
    is proven in
    \cite[Lemma 3.11]{sormani2016null}.
    We will show that
    \[
      \abs{f(\gamma(t_2))-f(\gamma(t_1))}
      \geq bL(\gamma|_{[t_1,t_2]}).
    \]

    We may assume that $\gamma([t_1,t_2])\subset U$ where
    $\phi:U\to \R^n$ is a chart. Let $\norm{\cdot}_{2}:\R^n\to\R$ be
    the Euclidean norm on $\R^n$. Throughout the proof below we
    freely identify differential forms in $T^*\R^n$
    with
    tangent vectors $T\R^n$ with points in $\R^n$ using the isomorphisms
    induced by the standard frame and metric on $\R^n$.
    The open Euclidean unit ball in $\R^n$ will be denoted $B$.
    In an abuse of notation, let $g$ and $g^{-1}$ denote the push forward
    of $g$ and $g^{-1}$ by $\phi$ onto $\phi(U)$.
    Since $g,g^{-1}$ are coordinate dependent metrics we must take some
    care with them when working in $\R^n$ with the identifications
    above.
    Note, in particular, 
    that by taking the points of evaluation, e.g.\ $x,y\in \phi(U)$,
    arbitrarily close we can ensure that $g^{-1}|_x$ is arbitrarily close
    to $g^{-1}|_y$.
    We have
    $f\circ\gamma=f\circ \phi^{-1}\circ \phi\circ\gamma$, so we can consider
    the function $f\circ\phi^{-1}$ on $\phi(U)$, which contains the
    curve $\phi\circ\gamma$.

    Since $\gamma([t_1,t_2])$ is compact and $U$ is open
    there exists $\epsilon^*>0$ so that
    the compact set
    \[
      V:=\{x\in\R^n: 
        \exists t\in[t_1,t_2],\, \norm{x-\gamma(t)}\leq \epsilon^*\}\subset
      \phi(U).
    \]

    The set-valued 
    Clarke generalised gradient of $f\circ \phi^{-1}$ at $y\in \phi(U)$
    is denoted $\partial^\circ(f\circ \phi^{-1})(y)$, see
    Definition \ref{def_cgg}.
    Minguzzi \cite[Theorem 1.19]{minguzzi2019lorentzian} shows that $f$ has a
    gradient defined almost everywhere on the manifold.
    Theorem \ref{thm:clarke} implies that, for each $y\in \phi(V)$, 
    vectors
    $v\in \partial^\circ(f\circ \phi^{-1})(y)$ satisfy
    $g^{-1}(v,v)\leq -b^2$, 
    because the past-pointing vectors  
    of Lorentzian length $\leq -b^2$ form a
    convex set.

    The function
    $f$ is locally Lipschitz and so there is a constant $D>0$ so that at
    any $y\in {\phi(V)}$ where $df$ exists we have 
    $\norm{df}_2\leq D$.
    Hence, again by Theorem \ref{thm:clarke}, any element 
    $v\in \partial^\circ(f\circ \phi^{-1})(y)$ 
    also satisfies $\norm{v}_2\leq D$.
    Thus we know that
    for each $y\in \phi(U)$ the set
    $K_y=\overline{\partial^\circ(f\circ \phi^{-1})(y+\epsilon^* B) + \epsilon^* B}$
    is compact.
    The set $K=\bigcup_{y\in V}K_y$ is a compact subset of
    $T\R^n$. Hence as $g$ is continuous
    there exists $C>0$ so that
    for all $y\in V$ and all $u,v\in K_y$,
    $g(u,v)\leq C$.

    Fix $\epsilon$ so that $0<\epsilon<\epsilon^*$
    and so
    that $-b^2 + (\epsilon^2 + 2\epsilon)C + \epsilon <0$.

    Since $\gamma$ is a continuous past directed causal curve we
    can, by \cite[Definition 7.4]{penrose1972techniques},
    approximate $\phi\circ\gamma|_{[t_1,t_2]}$ by a past directed
    causal curve which is a piecewise
    smooth $g$-geodesic 
    $\tilde{\gamma}:[t_1,t_2]\to\phi(V)$
    so that
    $\abs{L(\phi\circ\gamma|_{[t_1,t_2]}) - L(\tilde{\gamma})}<\epsilon$,
    $\tilde{\gamma}(t_1)=\phi\circ\gamma(t_1)$, 
    and
    $\tilde{\gamma}(t_2)=\phi\circ\gamma(t_2)$.
    Note that by construction, wherever $\tilde{\gamma}'$ exists
    it is a past directed $g$-causal vector.

    Since $g$ is smooth on $\phi(U)$,
    there exists a function $\delta:\phi(U)\to\R^+$ so that
    for all $x\in\phi(U)$,
    \begin{enumerate}
      \item $\delta(x)<\epsilon$,
      \item $\{y\in\R^n:\norm{x-y}_2\leq \delta(x)\}\subset \phi(U)$, and
      \item for all $x,y\in\phi(U)$,
        if $\norm{x-y}_2<\delta(x)$ 
        then 
        $\abs{g^{-1}|_x(d\tilde{f}(x),d\tilde{f}(x))
          -g^{-1}|_y(d\tilde{f}(x),d\tilde{f}(x))}<\delta(x)$.
    \end{enumerate}

    Since $f\circ \phi^{-1}$ is Lipschitz,
    Theorem \ref{thm_unreasonable_approx_result}
    implies that there exists 
    a smooth function $\tilde{f}:\phi(U)\to\R$ 
    such that for all $y\in \phi(U)$
    we have $\abs{f\circ \phi^{-1}(y)-\tilde{f}(y)}\leq\delta(y)<\epsilon$, 
    and
    \[
      d\tilde{f}(y)\in 
        \partial^\circ (f\circ\phi^{-1})((y + \delta(y) B)\cap \phi(U))
        + \delta(y) B.
    \]
    Note that while $d\tilde{f}$ is evaluated at $y$,
    the membership relation is for 
    the union of 
    Clarke's generalised gradients of $f\circ\phi^{-1}$ 
    over the set of points
    $(y + \delta(y) B)\cap \phi(U)$. 
    We want to use knowledge of 
    $\partial^\circ (f\circ\phi^{-1})(y)$ to construct a bound
    on $g(\nabla \tilde{f},\nabla\tilde{f})$. In doing so we
    will need to be careful about where the generalised gradient
    and the differential $d\tilde{f}$
    are evaluated.

    With these approximations we have
    \begin{align*}
      \abs{ \tilde{f}(\tilde{\gamma}(t_2))-\tilde{f}(\tilde{\gamma}(t_1)) }
      &=\abs{\tilde{f}(\phi\circ\gamma(t_2))-\tilde{f}(\phi\circ\gamma(t_1))}\\
      &=\abs{
        f(\gamma(t_2))+\tilde{f}(\phi\circ\gamma(t_2))-
        f(\gamma(t_2))-f(\gamma(t_1))-
        (\tilde{f}(\phi\circ\gamma(t_1))-f(\gamma(t_1)))}\\
      &< \abs{f(\gamma(t_2))-f(\gamma(t_1))}+2\epsilon.
    \end{align*}

    We now show that $x\in\phi(U)$ implies that
    $g^{-1}|_{x}(d\tilde{f}(x),d\tilde{f}(x))<0$, 
    thus showing
    that $\tilde{f}$ has $g$-causal gradient.
    By construction of $\tilde{f}$
    there exists 
    $y\in (x + \delta(x) B)\cap \phi(U)$
    so that
    \[
      d\tilde{f}(x)
      \in
        \partial^\circ (f\circ\phi^{-1})(y)+\delta(x) B.
    \]
    By definition there exists
    $w\in\partial^\circ f\circ \phi^{-1}(y)$
    and $u\in B$ so that
    $d\tilde{f}(x)=w + \delta(x)u$.
    We can compute that
    \begin{align*}
      g|_y(d\tilde{f}(x),d\tilde{f}(x)
        &=g|_y(w,w)
        +
        \delta(x)^2g|_y(u,u)
        +
        2
        \delta(\gamma(t)) g|_y(w,u)
        \leq 
        -b^2+(\epsilon^2+2\epsilon) C.
    \end{align*}
    Since $y\in (x + \delta(x) B)\cap \phi(U)$, by definition of
    $B$,
    we see that $\norm{x - y}_2<\delta(x)$
    and therefore know that
    $\abs{
      g^{-1}|_{x}(d\tilde{f}(x), d\tilde{f}(x))
      -
      g^{-1}|_y(d\tilde{f}(x),d\tilde{f}(x))}
      <
      \delta(x)$.
    Thus we have that
    \[
      g^{-1}|_{x}(d\tilde{f}(x), d\tilde{f}(x))
      \leq g^{-1}|_y(d\tilde{f}(x),d\tilde{f}(x)) + \delta(x)
      < -b^2+(\epsilon^2+2\epsilon) C + \epsilon
      < 0,
    \]
    by construction of $\epsilon$.
    Hence the $g$-gradient of $\tilde{f}$ is $g$-causal.

    By \cite[Prop 5.30]{oneil1983semi}, if $\tilde{\gamma}'$
    and $\grad \tilde{f}$ are both time-like
    the reverse Cauchy 
    inequality
    holds, 
    \[
      \abs{g(\grad \tilde{f}, \tilde{\gamma}')}
      \geq 
      \sqrt{-g(\grad \tilde{f}, \grad \tilde{f})}
      \sqrt{-g(\tilde{\gamma}',\tilde{\gamma}')},
    \]
    where $\nabla\tilde{f}$ is the gradient of $\tilde{f}$
    with respect to the push forward of $g$, the
    $g$-gradient.
    If $\grad \tilde{f}$ or $\tilde{\gamma}'$ is null then it is clear that this
    inequality continues to hold. Using this inequality we have 
    \begin{align*}
        \abs{\tilde{f}(\tilde{\gamma}(t_2))-\tilde{f}(\tilde{\gamma}(t_1))}
      &=
        -
        \int_{t_1}^{t_2}
          \frac{d}{dt}\tilde{f}(\tilde{\gamma}(t))
        \,\dd t
       =
        -
        \int_{t_1}^{t_2}
          d\tilde{f}(\tilde{\gamma}')
        \,\dd t
       = 
        \int_{t_1}^{t_2}
          \abs{g(\grad \tilde{f},\tilde{\gamma}')}
        \,\dd t\\
      &\geq 
        \int_{t_1}^{t_2}
          \sqrt{-g(\grad \tilde{f},\grad \tilde{f})}\,
          \sqrt{-g(\tilde{\gamma}',\tilde{\gamma}')}
        \,\dd t\\
      &> 
        \int_{t_1}^{t_2}
          \sqrt{b^2-(\epsilon^2+2\epsilon) C - \epsilon}\,
          \sqrt{-g(\tilde{\gamma}',\tilde{\gamma}')}
        \,\dd t\\
       &=
        \sqrt{b^2-(\epsilon^2+2\epsilon) C - \epsilon}\ 
         L(\tilde\gamma|_{[t_1,t_2]}).
    \end{align*}
    The result now follows by letting $\epsilon\to 0$.
	\end{proof}
	
  \begin{lemma}
    \label{lem_increasing_function_parameter}
    Let $(M,g)$ be a 
    Lorentzian manifold and
    let $f:M\to\R$ be a time function and $d(\cdot,\cdot; f)$ the associated
    null distance. 
    If $\gamma:[0,b)\to M$ is a past directed causal curve then
    there exists a change of parameter
    $s:[0, a)\to[0,b)$,
    where $a=\lim_{t\to b}(f\circ \gamma(0) - f\circ\gamma(t))$,
    so that
    $f\circ\gamma\circ s(t) = f\circ\gamma(0) - t$
    and, for all $t_1,t_2\in[0,a)$,
    \[
      d(\gamma\circ s(t_1), \gamma\circ s(t_2); f)=\abs{t_1-t_2}.
    \]
  \end{lemma}
  \begin{proof}
    Let $a=\lim_{t\to b} f(\gamma(0)) - f(\gamma(t))$.
    Since $\gamma_i$ is past directed and causal $f\circ\gamma_i$
    is a decreasing, continuous, and bijective function.
    It therefore has an inverse.
    Thus we can define $s:[0, a)\to [0, b)$ by
    $s(u) = (f\circ\gamma)^{-1}\left(f(\gamma(0)) - u\right)$.
    Note that $u = f(\gamma(0)) - f(\gamma\circ s(u))$.
    By construction $s$ is continuous, bijective and increasing.

    We can compute, by \cite[Lemma 3.11]{sormani2016null},
    that for all $u_1,u_2\in[0,a]$,
    \begin{align*}
      d(\gamma\circ s(u_1), \gamma\circ s(u_2); f)
      =
      \abs{f\circ\gamma\circ s(u_1) - f\circ\gamma\circ s(u_2)}
      =
      \abs{u_1-u_2}.
    \end{align*}
    Thus the result holds.
  \end{proof}
  
  The next result says that if the 
  range of $f$ and the parametrisation of the curves
  in a sequence
  by Lemma \ref{lem_increasing_function_parameter}
  are compatible
  then the domains of the
  curves converge.

  \begin{proposition}
    \label{lem_cauchy_sort_of}
    Let $(M,g)$ be a 
    Lorentzian manifold, $f:M\to\R$ be a locally Lipschitz,
    locally anti-Lipschitz time
    function and
    let $d(\cdot, \cdot\,;f)$ be
    the null distance associated to $f$, 
    Definition \ref{def_null_distance}.
    Let $(\gamma_i:[0,b_i)\to M)_{i\in\N}$ be a sequence of 
    past directed inextendible
    continuous causal curves in
    $M$, with $b_i\in\R^+\cup\{\infty\}$, such that 
    there exists $x\in M$ with $\gamma_i(0)\to x$.
    For each $i\in\N$ let
    $s_i:[0,a_i)\to[0,b_i)$ be the change in parameter
    constructed in Lemma \ref{lem_increasing_function_parameter}.
    Let $\gamma:[0,a)\to M$ be the limit curve constructed
    in Theorem \ref{thm_CurveUniformConvergenceInBoundedRegion_extension}
    from the sequence $(\gamma_i\circ s_i)_i$,
    and let $(\gamma_{i_k}\circ s_{i_k})_k$ be the subsequence of $(\gamma_i)$
    that
    $d(\cdot, \cdot; f)$-converges uniformly on compacta to $\gamma$,

    If $f$ is such that
    for all $r\in(-\infty, f(\gamma(0)))\cap\range{f}$
    we have $\gamma\cap f^{-1}(r)\neq\EmptySet$ and
    there exists $N\in\N$ so that for all
    $k\geq N$ we have $\gamma_{i_{k}}\cap f^{-1}(r)\neq\EmptySet$
    then
    there exists a subsequence $(\gamma_{i_{k_j}})_j$ of
    $(\gamma_{i_k})$ so that
    \[
      a=\limsup_{j} a_{i_{k_j}}.
    \]
  \end{proposition}
  \begin{proof}
    By Theorem \ref{thm_CurveUniformConvergenceInBoundedRegion_extension} we
    know that $a\leq\limsup_k a_{i_k}$.
    
    We first show that $f\circ\gamma(u)=f\circ\gamma(0) - u$.
    Let $u\in[0,a)$.
    Since $\gamma$ is past directed 
    $r=f\circ\gamma(u)\in(-\infty, f(\gamma(0))\cap\range{f}$.
    Thus by assumption there exists $N\in\N$ so that 
    $k\geq N$ implies that there exists $u_{i_k}\in[0,a_{i_k})$
    so that $f\circ\gamma_{i_j}\circ s_{i_j}(u_{i_j})=f\circ\gamma(u)$.
    The set $f^{-1}(r)$ is achronal and as each
    $\gamma_{i_k}$ and $\gamma$ are causal we know that 
    the sets
    $\gamma_{i_k}\cap f^{-1}(r)$ and $\gamma\cap f^{-1}(r)$ are
    singletons. Since $f$ is continuous $f^{-1}(r)$ is closed 
    and therefore, by the uniform convergence of $(\gamma_{i_k}\circ s_{i_k})$
    to $\gamma$, we know that $u_{i_k}\to u$.
    Thus
    \[
      f\circ\gamma(u)
      =f\circ\gamma_{i_k}\circ s_{i_k}(u_{i_k})
      =\lim_{k\to\infty}f\circ\gamma_{i_k}\circ s_{i_k}(u_{i_k})
      =\lim_{k\to \infty}
        f\circ\gamma_{i_k}(0) - u_{i_k}
      =f\circ\gamma(0) - u,
    \]
    as claimed.

    We now prove that $a=\limsup_ja_{i_j}$.
    Let $r\in(-\infty, f(x))\cap \range{f}$
    by assumption there
    exists
    $u\in[0,a)$ so that $f\circ\gamma(u)=r$.
    As $\gamma$ is past directed we now know that
    \[
      \inf\range{f}=\lim_{u\to a}f\circ\gamma(u)=
      f\circ\gamma(0) - \lim_{u\to a}u
      =f\circ\gamma(0) - a.
    \]
    By construction of each $s_i$, Lemma 
    \ref{lem_increasing_function_parameter},
    we know that
    $f\circ\gamma_{i_k}\circ s_{i_k}(t) = f\circ\gamma_{i_k}(0) - t$.
    By assumption
    there exists $N\in\N$ so that
    $k\geq K$ implies that there exists $u_{i_k}\in[0,a_{i_k})$
    so that $r=f\circ\gamma_{i_k}\circ s_{i_k}(u_{i_k})$.
    Thus as for $\gamma$,
    \[
      \inf\range{f}=f\circ\gamma_{i_j}(0) - a_{i_j}.
    \]
    By taking a subsequence, if necessary, we can assume that
    $\lim_k a_{i_{k}}=\limsup_k a_{i_k}$.
    Since $a\leq\limsup_ka_{i_k}$,
    \[
      \inf\range{f}=f\circ\gamma(0) - a
      \geq f\circ \gamma(0) - \limsup_ka_{i_k}
      = \lim_k f\circ \gamma_{i_{k}}(0) - a_{i_{k}}
      =\inf\range{f}.
    \]
    Thus $a=\limsup_k a_{i_k}$ as required.
  \end{proof}

  We now give the ``null distance limit curve theorem'' which
  demonstrates that for the null distance induced by a locally Lipschitz, locally
  anti-Lipschitz time function it is possible to have global length control
  over the limit curve.
  
  \begin{theorem}[Null distance limit curve theorem]
    \label{thm_reparametrisation_to_control_limsup}
    Let $(M,g)$ be a 
    Lorentzian manifold, $f:M\to\R$ be a locally Lipschitz,
    locally anti-Lipschitz time
    function and
    let $d(\cdot, \cdot\,;f)$ be
    the null distance associated to $f$, 
    Definition \ref{def_null_distance}.
    Assume that there is a constant $b>0$ such that for all points where
    $\nabla f$ exists we have $g(\nabla f,\nabla f)\leq -b^2$.
    
    Let $(\gamma_i:[0,b_i)\to M)_{i\in\N}$ be a sequence of 
    past directed inextendible
    continuous causal curves in
    $M$, with $b_i\in\R^+\cup\{\infty\}$, such that 
    there exists $x\in M$ with $\gamma_i(0)\to x$.
    For each $i\in\N$ let
    $s_i:[0,a_i)\to[0,b_i)$ be the change in parameter
    constructed in Lemma \ref{lem_increasing_function_parameter}.
    
    Let $\gamma:[0,a)\to M$ be the limit curve constructed
    in Theorem \ref{thm_CurveUniformConvergenceInBoundedRegion_extension}
    from the sequence $(\gamma_i\circ s_i)_i$,
    and let $(\gamma_{i_k}\circ s_{i_k})_k$ be the subsequence of $(\gamma_i)$
    that
    $d(\cdot, \cdot; f)$-converges uniformly on compacta to $\gamma$.
    If 
    \begin{enumerate}
      \item both sequences $(L(\gamma_{i_k}\circ s_{i_k}))_i$ and
        $(a_{i_k})_k$ are bounded above, and
      \item $\limsup_k a_{i_k}=a$,
    \end{enumerate}
    then
    there exists a subsequence $(\gamma_{i_{k_j}})_j$ of
    $(\gamma_{i_k})$ so that
    \[
      L(\gamma)\geq \limsup_j L(\gamma_{i_{k_j}}).
    \]
    Moreover if for all $t\in[0,a)$,
    \[
      L(\gamma|_{[0,t]})=\lim_kL(\gamma_{i_k}\circ s_{i_k}|_{[0,t]})
    \]
    then $L(\gamma)=\lim_kL(\gamma_{i_k}\circ s_{i_k})$.
  \end{theorem}
  \begin{proof}
    Since $f$ is locally anti-Lipschitz $d(\cdot, \cdot;f)$
    is a definite metric, and as $f$ is continuous $d(\cdot, \cdot; f)$
    is compatible with the manifold topology 
    as in Proposition \ref{prop:SV-metric}.
    Hence,
    Lemmas 
    \ref{lem_time_function_parametrisation}
    and
    \ref{lem_increasing_function_parameter}
    allow us to apply 
    Theorem	\ref{thm_CurveUniformConvergenceInBoundedRegion_extension}
    using the null distance of $f$
    to the sequence $(\gamma_i\circ s_i)$.

    To prove the theorem we need to use
    Lemma \ref{lem_improved_length_bound}.
    Let $B\in\R^+$ be such that for all $i\in\N$
    we have $a_i\leq B$. Thus
    $a\leq\limsup_i a_i\leq B$. In particular, $a$, and
    $\limsup_ia_i$
    are finite.
    By taking a subsequence we can 
    assume that $\lim_ka_{i_k}=\limsup_ka_{i_k}$.
    For each $i\in\N$ define $f_i:[0,a)\to[0,a_i)$ by
    $f_i(x)=a_i x/a$.

    For each $k\in \N$ we 
    have $\gamma_{i_k}\circ s_{i_k}\circ f_{i_k}:[0,c)\to M$.
    By assumption the sequence 
    $(L(\gamma_{i_k}\circ s_{i_k}\circ f_{i_k}))_i$ is bounded above,
    which implies that each
    $L(\gamma_{i_k}\circ s_{i_k}\circ f_{i_k})$ is finite.

    Choose $\epsilon>0$.
        Lemma \ref{lem_uni_for_repara} implies that 
        by taking a further subsequence we can arrange that
    $\sup_k\sup\{f_{i_k}'(t): t\in[0,a]\}\leq 1 + \epsilon$.
    By assumption and Lemmas \ref{lem_time_function_parametrisation},
    and \ref{lem_increasing_function_parameter}
    we therefore know that, for each $k\in\N$
    and all $t_1,t_2\in [0,a)$,
    \begin{align*}
      L\left(\gamma_{i_{k}}\circ s_{i_{k}}\circ f_{i_{k}}|_{[t_1,t_2]}\right)
      &=
      L\left(
        \gamma_{i_{k}}\circ s_{i_{k}}|_{[f_{i_{k}}(t_1),f_{i_{k}}(t_2)]}
      \right)\\
      &\leq
      \frac{1}{b}\abs{f_{i_{k}}(t_1)-f_{i_{k}}(t_2) }
      \leq \frac{1+\epsilon}{b}\abs{t_1-t_2}.
    \end{align*}
    Since, for all $k\in\N$ and all $t_1,t_2\in [0,a)$,
    we have
    \[
      L\left(\gamma_{i_{k}}\circ s_{i_{k}}\circ f_{i_{k}}|_{[t_1,t_2]}\right)
      \leq \frac{1+\epsilon}{b}\abs{t_1-t_2}
    \]
    we know that
    \[
      L\left(\gamma_{i_{k}}\circ s_{i_{k}}\circ f_{i_{k}}|_{[t_1,a)}\right)
      =
      \lim_{t_2\to a}
      L\left(\gamma_{i_{k}}\circ s_{i_{k}}\circ f_{i_{k}}|_{[t_1,t_2]}\right)
      \leq 
      \lim_{t_2\to a}
      \frac{1+\epsilon}{b}\abs{t_1-t_2}
      =
      \frac{1+\epsilon}{b}\abs{t_1-a}
    \]
    as $a$ is finite.

    Therefore to use Lemma \ref{lem_improved_length_bound} it remains
    to show that $(\gamma_{i_k}\circ s_{i_k}\circ f_{i_k})_k$
    converges uniformly to $\gamma$ with respect to $d(\cdot, \cdot; f)$
    on compact subsets of $[0, a)$. This is a consequence
    of $\lim_ka_{i_k}=a$.

    Let $C\subset [0, a)$ be compact and define $K=\sup C$.
    Choose $\tilde{\epsilon}>0$.
    Let $J_1\in\N$ be such that $k\geq J_1$ implies that
    \[
      K\abs{\frac{a_{i_k}}{a} - 1}<\tilde{\epsilon}.
    \]
    Since $(\gamma_{i_k}\circ s_{i_k}|_C)_k$ uniformly
    converges to $\gamma|_C$ there exists $J_2$ so that
    $k\geq J_2$ implies that,
    for all $c\in C$,
    \[
      d(\gamma_{i_k}\circ s_{i_k}(c), \gamma(c);f)<\tilde{\epsilon}.
    \]
    Hence if $k\geq \max\{J_1, J_2\}$ we can compute, for
    all $c\in C$, that
    \begin{align*}
        d\gamma_{i_k}\circ s_{i_k}\circ f_{i_k}(c), \gamma(c); f)
      &\leq
        d(
          \gamma_{i_k}\circ s_{i_k}\circ f_{i_k}(c), 
          \gamma_{i_k}\circ s_{i_k}(c); 
          f
        )\\
      &\hspace{2cm}+
        d(
          \gamma_{i_k}\circ s_{i_k}(c), 
          \gamma(c);
          f
        )\\
      &\leq
        \abs{f_{i_k}(c) - c}
        +
        \tilde{\epsilon}\\
      &\leq
        K\abs{\frac{a_{i_k}}{a} - 1}
        +
        \tilde{\epsilon}\\
      &<
        2\tilde{\epsilon}.
    \end{align*}
    That is, the conditions of 
    Lemma \ref{lem_improved_length_bound} hold
    and hence there exists a subsequence
    $(\gamma_{i_{k_j}}\circ s_{i_{k_j}}\circ f_{i_{k_j}})_j$
    of
    $(\gamma_{i_{k}}\circ s_{i_{k}}\circ f_{i_{k}})_k$
    so that
    \[
      L(\gamma)\geq \limsup_j L(\gamma_{i_{k_j}}\circ s_{i_{k_j}}\circ f_{i_{k_j}})
      = \limsup_j L(\gamma_{i_{k_j}}\circ s_{i_{k_j}}),
    \]
    as required.

    Suppose further that
    for all $t\in[0,a)$,
    \[
      L(\gamma|_{[0,t]})=\lim_kL(\gamma_{i_k}\circ s_{i_k}|_{[0,t]}).
    \]
    Supposing that for all $k\in \N$ we have $f_{i_k}(t)> t$,
    \begin{align*}
      0\leq \lim_k L(\gamma_{i_k}\circ s_{i_k}|_{[t,f_{i_k}(t)]})
      \leq 
      \lim_k \frac{1}{b}\abs{t - \frac{a_{i_k}}{a}t}
      = 0,
    \end{align*}
    hence
    \begin{align*}
      L(\gamma|_{[0,t]})
      =
      \lim_k L(\gamma_{i_k}\circ s_{i_k}|_{[0,t]})
      =
      \lim_k L(\gamma_{i_k}\circ s_{i_k}\circ f_{i_k}|_{[0,t]})
    \end{align*}
    and hence by 
    Lemma \ref{lem_improved_length_bound},
    \[
      L(\gamma)=\lim_jL(\gamma_{i_{k_j}}\circ s_{i_{k_j}}\circ f_{i_{k_j}})
      =\lim_kL(\gamma_{i_{k_j}}\circ s_{i_{k_j}}).
    \]

    A similar argument holds if
    for all $k\in \N$ we have $f_{i_k}(t)\leq t$.
    Since we can restrict to a subsequence so that one of these
    two conditions hold the theorem is true.
  \end{proof}
  
  The global bounds on lengths in Theorem
  \ref{thm_reparametrisation_to_control_limsup} can be achieved for space-times
  whose past is bounded.
  
  \begin{theorem}
  \label{thm:reg-cosmo}
    If $(M,g)$ is a globally hyperbolic manifold with
    regular cosmological time $\tau$, then for any sequence 
    $(\gamma_i:[0,b_i)\to M)$, $b_i\in\R^+\cup\{\infty\}$,
    of
    past directed inextendible continuous causal curves
    so that $\gamma_i(0)\to x\in M$ conditions 1.\ and 2.\
    of Theorem \ref{thm_reparametrisation_to_control_limsup} hold.
  \end{theorem}
  \begin{proof}
    By \cite[Theorem 1.2(2 and 5)]{andersson1998cosmological}
    the regular cosmological time is continuous and locally Lipschitz.
    By \cite[Theorem 5.4]{sormani2016null} the 
    regular cosmological time is anti-Lipschitz.

    Since the cosmological time $\tau$ is regular
    on any past directed inextendible curve, $\gamma:[0,b_i)\to M$,
    we know that $\lim_{t\to b_i}\tau\circ\gamma(t)=0$.
    This implies that the cosmological time satisfies the
    conditions of Proposition \ref{lem_cauchy_sort_of}.
    As the cosmological time is regular it is also finite, hence
    the result holds.
  \end{proof}

\section{Null distances of surface functions}
  \label{subsec_surface_function_generalities}

  To apply
  Theorem \ref{thm_reparametrisation_to_control_limsup}
  we need
  a function $f:M\to\R$ that is
  \begin{enumerate}[noitemsep]
    \item locally Lipschitz, 
    \item locally anti-Lipschitz, 
    \item a time function, and such that
    \item there exists some $b\in\R^+$ so that $g(\nabla f, \nabla f)\leq -b^2$
      wherever $\nabla f$ exists, 
  \end{enumerate}
  and from Proposition 
  \ref{lem_cauchy_sort_of}, the sequence of curves needs to satisfy:
  for all $r\in(-\infty, f(\gamma(0)))\cap\range{f}$
  we have $\gamma\cap f^{-1}(r)\neq\EmptySet$ and
  there exists $N\in\N$ so that for all
  $k\geq N$ we have $\gamma_{i_{k}}\cap f^{-1}(r)\neq\EmptySet$.
  
  We shall call the first four of these conditions the regularity conditions
  and we shall call the last condition the geometric condition.
  Theorem \ref{thm:reg-cosmo} proves that regular cosmological
  time satisfies both the regularity and the geometric conditions.

  The geometric condition depends on the given sequence of inextendible curves,
  and so the function $f$ could in principle  be tailored for application to a
  specific sequence of curves.
  The geometric condition is always true if the level sets of $f$ are Cauchy.
  In the remainder of this section we shall show that the surface function
  associated to a $C^1$ Cauchy surface satisfies the regularity conditions,
  Corollary \ref{corl:surf}.

  Before we prove our claim it is worth noting that
  M\"uller and S\'anchez, \cite[Theorem 1.2]{muller2011lorentzian},
  have shown the existence of a smooth surjective
  Cauchy time function 
  $f:M\to\R$ 
  on any globally
  hyperbolic manifold
  so that $g(\nabla f, \nabla f)\leq -1$.
  Since $f$ is Cauchy it satisfies the geometric condition.
  Since the function is smooth it is locally Lipschitz.
  Sormani and Vega have shown, \cite[Corollary 4.16]{sormani2016null},
  that such a function is also anti-Lipschitz.
  Thus
  M\"uller and S\'anchez' time function also satisfies the regularity conditions
  except for the bound on gradient length.

  If one could prove that a lower bound for
  $g(\nabla f, \nabla f)$ existed then 
  M\"uller and S\'anchez's time function would satisfy both the
  regularity and the geometric conditions.
  The existence of such a lower bound is tied to the geometry of the manifold.
  In truncated Minkowski space, i.e.\ $\R\times(0,\infty)$, no such
  bound exists since $f$ surjects onto $\R$, though, of course, the standard time 
  function given by projection does satisfy the regularity 
  and geometric conditions.

  We remind the reader that our Cauchy surfaces are, by definition,
  acausal \cite[page 65]{beem1996global}.

  We shall rely on the following technical result. 
  \begin{lemma}
    \label{lem_basic_gh_curve_existence_result}
    Let $(M,g)$ be globally hyperbolic, let $S\subset M$ be a Cauchy surface, 
    let $\tau_S:M\to\R$ the surface function of $S$,
    let $x\in \future{S}$
    and let $h$ be an auxiliary Riemannian
    metric.
    Let $(x_i)_{i\in\N}$ be a sequence of points so that 
    $(x_i)\subset \future{x}$ and $x_{i+1}\in\past{x_i}$,
    $x_i\to x$. 
    If $(\gamma_i:[0,b_i]\to M)_{i\in\N}$, $b_i\in\R^+$,
    is a sequence
    of $h$-arc length parametrised, past directed,
    causal curves so that for all $i\in\N$, 
    $\gamma_i(0)=x_i$, $\gamma_i(b_i)\in S$,
    and $L(\gamma_i)\geq \tau_S(x_i)-1/i$, then
    there exists 
    an $h$-arc length parametrised,
    past directed,
    timelike curve $\gamma:[0,a]\to M$, $a\in\R^+$,
    so that  $\gamma(0)=x$ and $\gamma(a)\in S$
    and is such that
    \begin{enumerate}
      \item there exists
        a subsequence $(\gamma_{i_k})_k$ 
        of $(\gamma_i)_i$ which
        $d(\cdot,\cdot;h)$-converges to $\gamma$ uniformly on compacta,
      \item $\tau_S(x) = L(\gamma) 
        = d_L(\gamma(a),\gamma(0)) = \lim_k\tau_S(x_{i_k})$, and
      \item $\gamma$ can be re-parametrised as a smooth timelike geodesic.
    \end{enumerate}
  \end{lemma}
  \begin{proof}
    Choose $S'$ a second Cauchy surface so that
    $S'\subset\past{S}$.
    Since $S'$ is Cauchy we can extend each $\gamma_i$
    to $\tilde{\gamma}_i\in\Omega_{S', x_i}$.
    We will parametrise each extended $\tilde{\gamma}_i$ 
    by $h$ induced arc-length.
    Thus we have
    $\tilde{\gamma}_i:[0, a_i)\to M$, 
    $a_i\in(0,\infty]$,
    so that $L(\tilde{\gamma}_i\cap\future{S})\geq \tau_{S}(x_i)-1/i$
    and for all $i\in\N$ and $t_1, t_2\in [0,a_i)$ we have
    $d(\tilde{\gamma}_i(t_1),\tilde{\gamma}_i(t_2);h)\leq \abs{t_1-t_2}$.

    Choose $y\in\future{x}$.
    Thus $y\in\future{x}\subset\future{S}\subset\future{S'}=
    \interior{D^+(S')}$
    as $S'$ is achronal, by definition,
    \cite[Proposition 5.20]{penrose1972techniques}
    implies that
    $\jpast{y}\cap\jfuture{S'}$ is compact.
    As $\jpast{y}\cap\jfuture{S'}$ is compact
    and as each $\tilde{\gamma}_i$ is extendible, we see that in fact the $a_i$
    are finite and indeed there
    exists $b\in\R^+$ so that $a_i\leq b$, 
    \cite[Theorem 1.35]{minguzzi2019lorentzian}.

    Since $\jpast{y}\cap\jfuture{S'}$ is compact, 
    Proposition 
    \ref{prop_CurveUniformConvergenceInBoundedRegion}
    implies that there exists 
    \begin{enumerate}
      \item a continuous curve $\gamma:[0,a)\to M$, 
        $a=\limsup_ia_i<\infty$,
      \item a sequence of re-parametrisations $f_i:[0,a)\to [0, a_i)$,
      \item a subsequence $(\gamma_{i_k})_k$ of $(\gamma_i)$ so that
        $(\gamma_{i_k}\circ f_{i_k})_k$ will
        $d(\cdot,\cdot;h)$-converge uniformly 
        on compact subsets of $[0,a)$ to
        $\gamma$ and so that $\lim_k a_{i_k}=a$.
    \end{enumerate}

    Since for all $i\in\N$, $a_i<\infty$ and $a<\infty$
    we know that $f_i(x)=a_ix/a$, Lemma \ref{lem_uni_for_repara}.
    As $a=\lim_ka_{i_k}$
    we know that there exists $K\in\R^+$
    so that $\sup\{f_i'(t):t\in[0,a]\}=a_i/a \leq K$.
    This implies that $(\gamma_{i_k}\circ f_{i_k})_k$ will
    $d(\cdot,\cdot;h)$-converge uniformly to
    $\gamma$ on compacta, in the sense of Definition
    \ref{defn:limit-curve}.

    Since for all $i\in \N$, $\gamma_i\cap S\neq\EmptySet$
    we see that $\gamma\cap S\neq\EmptySet$. Thus
    there exists $t\in[0,a)$ so that
    $\gamma(t)\in S$. Similarly for each $k$ there exists $t_k\in[0,a)$ so
    that $\gamma_{i_k}\circ f_{i_k}(t_k)\in S$. 
    Similarly to the proof that $\jpast{y}\cap\jfuture{S'}$ is compact,
    \cite[Proposition 5.20]{penrose1972techniques}
    implies that
    the set 
    $\jpast{y}\cap \jfuture{S}$ is compact.
    Hence $t_k\to t$ by the uniform convergence of
    the $(\gamma_{i_k})$.
    By Proposition
    \ref{thm_upper_semi_cts}
    \[
      L(\gamma|_{[0,t]})\geq 
        \limsup_k L(\gamma_{i_k}\circ f_{i_k}|_{[0,t]})
        =\limsup_k L(\gamma_{i_k}\circ f_{i_k}|_{[0,t_k]}).
    \]
    Since $x_{i_k}\in \future{x}$ we have $\tau_S(x_{i_k})\geq\tau_S(x)$.
    Hence we also have $\limsup_k\tau_S(x_{i_k})\geq\tau_S(x)$.
    Therefore 
    \begin{align}
      \limsup_k\tau_S(x_{i_k})
      &\geq
      \tau_S(x)
      \geq 
      L(\gamma|_{[0,t]})\nonumber \\
      &\geq 
      \limsup_k L(\gamma_{i_k}\circ f_{i_k}|_{[0,t]})
      =
      \limsup_k L(\gamma_{i_k}\circ f_{i_k}|_{[0,t_k]})\nonumber \\
      &\geq
      \limsup_k \tau_S(x_{i_k}) - 1/{i_k}\nonumber \\
      &=
      \limsup_k \tau_S(x_{i_k}).
    \label{eq:one-way}
    \end{align}
    That is, we have shown that
    $\tau_S(x)=L(\gamma|_{[0,t]})=\limsup_k\tau_S(x_{i_k})$. 
    Since $x_{i+1}\in\past{x_i}$
    the sequence $(\tau_S(x_{i_{k}}))_k$ is decreasing.
    Thus $\limsup_k\tau_S(x_{i_k})=\lim_k\tau_S(x_{i_k})$.
    Since 
    $
      d_L(\gamma(t),\gamma(0))\geq L(\gamma|_{[0,t]})=\tau_S(\gamma(0))\geq
      d_L(\gamma(t), \gamma(0)),
    $ 
    Theorem 4.13 of 
    \cite{beem1996global}
    implies that
    $\gamma|_{[0,t]}$ can be reparametrised as a smooth timelike geodesic.
    This implies that $\gamma|_{[0,t]}$ is a timelike curve
    and so gives the result.
  \end{proof}

  As a corollary we get the following result.
  \begin{corollary}
    \label{lem_surface_function_exists}
    If $(M,g)$ is globally hyperbolic and $S\subset M$ is a Cauchy surface
    then $M$ is the disjoint union
    $M=\future{S}\cup S\cup \past{S}$. Moreover
    for all $x\in \future{S}$, $d_L(S, x)<\infty$
    and for all $x\in\past{S}$, $d_L(x,S)<\infty$.
  \end{corollary}
  \begin{proof}
    By assumption and as $S$ is Cauchy the interior of $D^+(S)$ is
    $\future{S}$. Since $S$ is acausal 
    we have
    that $M$ is the disjoint union
    $M=\future{S}\cup S\cup \past{S}$. 
    Lemma \ref{lem_basic_gh_curve_existence_result}
    implies that there exists a curve with compact
    image which attains the distance to the
    surface. This implies that the distance to the surface is
    finite. The result follows by time duality.
  \end{proof}
    
  \begin{lemma}
    \label{lem_tau_gh_is_cts}
    If $(M,g)$ is globally hyperbolic and $S\subset M$ is a Cauchy surface,
    then $\tau_S$ is continuous.
  \end{lemma}
  \begin{proof}
    Suppose that $\tau_S$ is discontinuous at $x\in\overline{\future{S}}$.
    That is, there exists a sequence $(y_i)_{i\in\N}\subset M$
    with $y_i\to x$ and $\lim_{i\to\infty}\tau_S(y_i)>\tau_S(x)$.
    Choose $(x_i)_i\in\N\subset\future{x}$ so that
    $x_{i+1}\subset\past{x_i}$ and $x_i\to x$.
    Since $x\in\past{x_i}$ we see that
    $\lim_{i\to\infty}\tau_S(x_i)\geq\lim_{i\to\infty}\tau_S(y_i)>\tau_S(x)$.
    We can now find curves $\gamma_i\in \Omega_{S,x_i}$ with
    $\tau_S(x_i)=L(\gamma_i)-1/i$, by definition of $\tau_S$. 
    Consequently the existence of such a sequence of points $(y_i)$ contradicts
    Lemma \ref{lem_basic_gh_curve_existence_result}, hence we have the result.
  \end{proof}

  \begin{proposition}
    \label{prop_tau_prop}
    Let $(M,g)$ be globally hyperbolic, let $S\subset M$ be a Cauchy surface, and
    $\tau_S:M\to\R$ the surface function of $S$.
    For all $x\in\future{S}$ there exists $\gamma:[0,a]\to M$, $a\in\R^+$, a 
    future directed, $g$-arc length parametrised,
    timelike smooth geodesic from $S$ to $x$ so that
    \begin{enumerate}
      \item for all $u,v\in[0,a]$, $u<v$, 
        $L(\gamma|_{[u,v]})=
          d_L(\gamma(u),\gamma(v))=\tau_S(\gamma(v))-\tau_S(\gamma(u))$;
      \item for all $u\in[0,a]$,
        if $\nabla\tau_S$ exists at $\gamma(u)$
        then $\gamma'(u)=-\nabla\tau_S|_{\gamma(u)}$ and, in particular,
        \[
          g(\nabla\tau_S|_{\gamma(u)},\nabla\tau_S|_{\gamma(u)})=-1.
        \]
    \end{enumerate}
  \end{proposition}
  \begin{proof}
    Choose a sequence $(x_i)_i\subset\future{x}$
    so that $x_{i+1}\subset\past{x_i}$.
    For 
    each $i\in\N$ choose
    $\gamma_i:[0,b_i]\to M$, $b_i\in\R^+$, be a past directed causal
    curve from $x_i$ to $S$ so that
    $L(\gamma_i)\geq \tau_S(x_i) - 1/i$.
    Let $\gamma:[0,\tilde{a}]\to M$ be the curve constructed by
    Lemma \ref{lem_basic_gh_curve_existence_result}
    using the sequence of curves $(\gamma_i)_i$.
    We re-parametrise $\gamma$ to be arc-length
    parametrised with respect to $g$.
    In an abuse of notation we denote this re-parametrised
    curve by $\gamma:[0,a]\to M$, $a\in\R^+$.
    Thus, by definition $g(\gamma',\gamma')=-1$.

    Remark 4.11 of \cite{beem1996global}
    shows that 
    for all $u,v\in[0,a]$, $u<v$, 
    $L(\gamma|_{[u,v]})=d_L(\gamma(u),\gamma(v))$.
    Suppose that for some $u\in[0,a]$ we
    have $\tau_S(\gamma(u))>L(\gamma|_{[0,u]})$.
    Then
    \[
      L(\gamma)=\tau_S(x)\geq d_L(\gamma(u), x) + \tau_S(\gamma(u))
        > L(\gamma|_{[u,a]}) + L(\gamma|_{[0,u]})
        =L(\gamma),
    \]
    which is a contradiction.
    We find that
    for all $u\in[0,a]$ we have $\tau_S(\gamma(u))=L(\gamma|_{[0,u]})$.
    Hence, for all $u, v\in[0,a]$, $u<v$
    \[
      L(\gamma|_{[0,v]})=\tau_S(\gamma(v))\geq 
        d_L(\gamma(u), \gamma(v)) + \tau_S(\gamma(u))
        = L(\gamma|_{[u,v]}) + L(\gamma|_{[0,u]})
        =L(\gamma|_{[0,v]}),
    \]
    and therefore
    \[
      L(\gamma|_{[u,v]})=
          d_L(\gamma(u),\gamma(v))=\tau_S(\gamma(v))-\tau_S(\gamma(u)).
    \]

    Since $\gamma$ is arc length parametrised,
    for all $u\in[0,a]$ we
    have $\tau_S(\gamma(u))-\tau_s(\gamma(0))=L(\gamma|_{[0,u]})=u$.
    Thus, by the reverse Cauchy inequality,
    \cite[Proposition 5.30]{oneil1983semi},
    and the direct calculation in \cite[Propostion 2.16]{rennie2016generalised},
    wherever $\nabla\tau_S$ exists we have
    \[
      1=g(\nabla\tau_S,\gamma')\geq
        \sqrt{-g(\nabla\tau_S,\nabla\tau_S)}
        \sqrt{-g(\gamma',\gamma')}\geq 1
    \]
    and so $\nabla\tau_s=-\gamma'$
    and $g(\nabla\tau_S,\nabla\tau_S)=-1$.
  \end{proof}
    
  \begin{lemma}
    \label{lem_tau_gh_is_acausal}
    If $(M,g)$ is globally hyperbolic and $S\subset M$ is a
    Cauchy surface,
    then for all $t\in\range{\tau_S}$
    the set
    $\tau_S^{-1}{(t)}$ is acausal.
  \end{lemma}
  \begin{proof}
    Suppose that $t>0$.
    Let $x,y\in\tau_S^{-1}(t)$ be such that $y\in\jfuture{x}\setminus\{x\}$. 
    As $\tau_S$ is increasing on
    timelike curves we know that $y\not\in\future{x}$.
    Therefore there exists $\gamma$ a null curve from $x$ to $y$
    on which $\tau_S$ is constant.

    Proposition \ref{prop_tau_prop}
    implies that there exists $\lambda\in\Omega_{S,x}$
    a timelike geodesic so that
    $L(\lambda)=\tau_S(x)$. 

    The concatenation of $\gamma$ and $\lambda$, denoted $\sigma$, 
    is a causal curve
    so that $L(\sigma)=\tau_S(y)$.
    Theorem 4.13 of
    \cite{beem1996global}
    implies that $\sigma$ can be reparametrised
    as a smooth timelike geodesic. 
    This is a contradiction as, at $x$,
    $g(\gamma',\gamma')=0$ and $g(\lambda',\lambda')<0$.

    The result now follows by time duality
    and as our Cauchy surfaces are necessarily acausal
    \cite[Page 65]{beem1996global}.
  \end{proof}

    \begin{lemma}
      \label{lem_gh_sur_gtf}
      Surface functions of Cauchy surfaces in globally hyperbolic
      manifolds are generalised time functions,
      i.e.\ increasing on all future directed causal curves.
    \end{lemma}
    \begin{proof}
      Let $\tau_S$ be the surface function of the Cauchy surface $S$. 
      Let $\gamma$ be a future directed causal curve.
      We know that $\tau_S$ is non-decreasing on $\gamma$.
      If $\tau_S$ is constant over some subcurve of $\gamma$
      then there exists a level surface of $\tau_S$
      that contains a causal curve.
      This contradicts
      Lemma \ref{lem_tau_gh_is_acausal} and the assumption that
      $S$ is acausal. Thus $\tau_S$ is a non-constant non-decreasing
      function on $\gamma$ and therefore is increasing.
    \end{proof}

    \begin{lemma}
      \label{lem_tauS_c1_near_S}
      If $(M,g)$ is globally hyperbolic and $S\subset M$ is
      a $C^1$ Cauchy surface,
      then there exists an open 
      neighbourhood $U\subset M$ of $S$ 
      and a continuous vector field $X:U\to TM$ so that
      wherever $\grad\tau_S$ exists $X=\grad\tau_S$.
    \end{lemma}
    \begin{proof}
      As $S$ is $C^1$ we know that the unit normal vector field $n$ to $S$
      is continuous and timelike everywhere. Let $NS$ denote the normal bundle
      to $S$,
      that is $N_sS$ is the span of $n(s)\in T_sM$ for each $s\in S$.
      Let $V\subset NS$ be an open neighbourhood of
      the zero section of $NS$ over $S$
      so that $\exp(V)$ is a normal neighbourhood of $S$.

      Let $s\in S$ and 
      let $I_s=\{t\in\R: tn(s)\in V\}$.
      Define $\gamma_s(t):I_s\to M$ by
      $\gamma_s=\exp_s(tn(s))$.
      By definition $\gamma_s$ is focal point free
      and therefore 
      for all $t\in I_s\cap\R^+$, 
      $t=L(\gamma_s|_{[0,t]})=d_L(S, \gamma_s(t))=\tau_S(\gamma_s(t))$,
      see \cite[Propositions 12.25, 12.29]{beem1996global}.
      Similarly,
      for all $t\in I_s\cap\R^-$, 
      $-t=L(\gamma_s|_{[t,0]})=d_L(\gamma_s(t), S)= -\tau_S(\gamma_s(t))$.
      Thus $\tau_S(\gamma_s(t))=t$ for all $t\in I_s$.
      Therefore, if $\grad\tau_S|_{\gamma_s(t)}$ exists
      then $\grad\tau_S|_{\gamma_s(t)}=-\gamma'_s(t)$.
      Let $\partial_t\in TI_s$ be the standard unit length tangent vector.
      Thus $\gamma_s'(t) = (d\exp_s(\partial_t))|_{tn(s)}$.
      Let $X(\exp(tn(s)))= -\gamma_s'(t)=-(d\exp_s(\partial_t))|_{tn(s)}$.
      As $n$ is continuous and $d\exp$ is smooth
      $X$ is a continuous vector field so that
      if $\grad\tau_S$ exists at $x$ then 
      $X(x)=\grad\tau_S$. 
    \end{proof}

    \begin{lemma}
      \label{lem_tau_S_bounded_h_grad}
      Let $(M,g)$ be globally hyperbolic, $\tau_S$ the surface function
      associated to a $C^1$ Cauchy surface $S\subset M$ and $h$ an auxiliary
      Riemannian metric.
      For all $C\subset M$, a compact subset, there exists $K\in\R^+$
      so that for all $c\in C$ where $\grad\tau_S|_c=g(d\tau_S|_c, \cdot)$
      exists we have $h(\grad\tau_S|_c, \grad\tau_S|_c)< K$.
    \end{lemma}
       \begin{proof}
      Lemma \ref{lem_tauS_c1_near_S} implies that
      there exists a neighbourhood $U\subset M$ of $S$
      and a continuous vector field $X:U\to TM$ so that
      $X=\grad\tau_S$ wherever $\grad\tau_S|_U$ exists.
      If we suppose that $C\subset U$, then since $h(X,X)$ is continuous and
      $C$ is compact the
      result holds. 
    
      More generally, we start by considering $C\cap U$. Let
      $\phi_+,\phi_-,\phi_0$ be a partition of unity whose supports are in
      $\future{S}$, $\past{S}$ and $U$ respectively. Let $C_0=\{x\in C\cap
      U:\,\phi_0\geq 1/2\}$. Then $C_0$ is a closed subset of the compact set
      $C$, and so is compact, and $C_0\subset U$. Now define $C_+$ to be the
      closure of $(C\cap \future{S})\setminus C_0$, and 
      similarly set $C_-$ to be the closure of $(C\cap \past{S})\setminus C_0$.
      
      We have already shown that the result is true for $C_0$. Therefore, by
      time duality,
      if the result holds for any compact subset
      $C\subset\future{S}$ then the result will hold for
      any compact subset $C\subset M$.
    
      With this in mind suppose that
      there exists a compact set $C\subset\future{S}$ and
      $(x_i)\subset C$ a sequence of points in $C$
      so that $\grad\tau_S|_{x_i}$ exists and
      $\lim_{i\to\infty}h(\grad\tau_S|_{x_i}, \grad\tau_S|_{x_i})=\infty$.
      By taking a subsequence we can further assume that $x_i\to x\in C$.

      Since $M$ is globally hyperbolic $A=\jpast{C}\cap \jfuture{S}$
      is compact. Thus, by Proposition
      \ref{prop_tau_prop},
      for each $i\in\N$ there exists
      a smooth past directed timelike geodesic
      $\gamma_i:[0,a_i]\to A$ so that $\gamma_i(0)=x_i$,
      $\gamma_i(a_i)\in S$ and $L(\gamma_i)=\tau_S(x_i)$.
      Without loss of generality we can assume that each 
      $\gamma_i$ is $h$ arc length parametrised.
      That is, we can assume that
      $h(\gamma_i',\gamma_i')=1$.

      Let $\gamma:[0,a)\to A$ be the smooth timelike geodesic from
      $x$ given by applying 
      Lemma
      \ref{lem_basic_gh_curve_existence_result}
      to the sequence $(\gamma_i)$.
      Lemma
      \ref{lem_basic_gh_curve_existence_result}
      tells us that
      $\tau_S(x)=L(\gamma)=\lim_k \tau_S(x_{i_k})$. In particular, $\gamma$ is
      timelike.

      Due to our parametrisation,
      Proposition
      \ref{prop_tau_prop} implies that
      \[
        \gamma_i'(0)=\frac{1}{\sqrt{h(\grad\tau_S|_{x_i},\grad\tau_S|_{x_i})}}
          \grad\tau_S|_{x_i} .
      \]
      By taking a convex normal neighbourhood about $x$
      we see that there exists $\tau\in\R^+$ so that
      for $i$ large enough $\tau$ is in the domain
      of $\gamma_i$.
      Since $\gamma_i$ is a geodesic we know that
      $\gamma_i(t)=\exp_{\gamma_i(0)}(t\gamma_i'(0))$, at least for
      $t\in[0,\tau]$.
      The world function, \cite[Definition 2.13]{penrose1972techniques},
      on $U$, $\Phi:U\times U\to\R$ is defined by
      \[
        \Phi(p,q)=g(\exp_p^{-1}(q), \exp_p^{-1}(q)).
      \]
      We can compute that
      \[
        \Phi(\gamma_i(0),\gamma_i(t))=t^2g(\gamma_i'(0),\gamma_i'(0))
        =-\frac{t^2}{{h(\grad\tau_S|_{x_i},\grad\tau_S|_{x_i})}}.
      \]
      Since $\Phi$ is continuous, \cite[Definition 2.13]{penrose1972techniques}
      we see that 
      \begin{equation}
        \label{eq_locally_bounded_zero}
        \Phi(\gamma(0),\gamma(t))
        =
        \lim_{i\to\infty}
        \Phi(\gamma_i(0),\gamma_i(t))
        =
        \lim_{i\to\infty}
        t^2g(\gamma_i'(0),\gamma_i'(0))
        =
        \lim_{i\to\infty}
        -\frac{t^2}{{h(\grad\tau_S|_{x_i},\grad\tau_S|_{x_i})}}
        =0.
      \end{equation}
      As $\gamma$ is a timelike geodesic we know that
      $\Phi(\gamma(0),\gamma(t))<0$, \cite[Lemma 2.15]{penrose1972techniques}, so
      Equation \eqref{eq_locally_bounded_zero} gives us a contradiction.
      Hence the result holds for $C\subset\future{S}$ and
      thus for any compact subset of $M$.
    \end{proof}

    \begin{corollary}
      \label{cor_grad_lives_in_compact_Set}
      Let $(M,g)$ be globally hyperbolic, $\tau_S$ the surface function
      associated to a $C^1$ Cauchy surface $S\subset M$ and $h$ an auxiliary
      Riemannian metric.
      If $C\subset M$ is compact then
      \[
        \{\grad\tau_S|_c: c\in C,\  \grad\tau_S\text{ exists at } c\}
          \subset
        \{v\in T_cC: g(v,v)=-1,\ h(v,v)\leq K\},
      \]
      which is a compact subset of the tangent bundle $TC$.
    \end{corollary}
    \begin{proof}
      This is an immediate consequence of Lemma 
      \ref{lem_tau_S_bounded_h_grad}
      and Proposition \ref{prop_tau_prop}.
    \end{proof}

    \begin{corollary}
      \label{cor_sf_alip}
      Let $(M,g)$ be globally hyperbolic. 
      If $\tau_S$ is the surface function
      associated to a $C^1$ Cauchy surface, then
      $\tau_S$ is locally anti-Lipschitz with respect to 
      any auxiliary
      Riemannian metric $h$.
    \end{corollary}
    \begin{proof}
      Let $p\in M$, $C\subset M$ be a compact neighbourhood
      of $p$ and let $U\subset C$ be an open neighbourhood of $p$.
      Lemma \ref{lem_tau_S_bounded_h_grad}
      implies that there exists $K\in\R^+$ so that
      wherever $\grad\tau_S|_C$ exists
      $h(\grad\tau_S,\grad\tau_S)< K$.
      We also know that $g(\grad\tau_S,\grad\tau_S)=-1$, by
      Proposition \ref{prop_tau_prop}.
      Therefore 
      \[
        \sqrt{\abs{g(\grad\tau_S,\grad\tau_S))}}
        = 1 \geq \frac{1}{\max\{1,K\}}\max\{1, \sqrt{h(\grad\tau_S,\grad\tau_S)}\},
      \]
      and so $\grad\tau_S$ is bounded away from light cones.
      The 
      local anti-Lipschitz property now follows from
      \cite[Definition 4.13 and Theorem 4.18]{sormani2016null}.
    \end{proof}

    \begin{lemma}
      \label{lem_sf_llip}
      Let $M$ be globally hyperbolic. 
      If $\tau_S$ is the surface function
      associated to a $C^1$ Cauchy surface, then
      $\tau_S$ is locally Lipschitz with respect to
      any auxiliary
      Riemannian metric $h$.
    \end{lemma}
    
    \begin{proof}
      Since the statement is local, and all Riemannian metrics are locally
      equivalent, it suffices to check local Lipschitzness with respect to a
      single metric.
    
      Let $x\in M$ and let $\phi:U\to \R^n$ be a chart about $x$.
      Let $\partial_1,\ldots,\partial_n$ be the coordinate
      frame on $U$.
      Choose $e_1,\ldots,e_n$ a pseudo-orthonormal frame over $U$,
      with $g(e_1,e_1)=-1$, $g(e_i,e_j)=\delta_{ij}$
      for $i=1,\ldots, n$ and $j=2,\ldots, n$.
      Define $h:TU\times TU\to\R^+$ a Riemannian metric
      \[
        h(u,v)=g(u,v) - 2\frac{g(u,e_1)g(v,e_1)}{g(e_1,e_1)},
      \]
      so that $h_{ij}=\abs{g_{ij}}=\delta_{ij}$.

      For any open set $V\subset\R^n$
      let $L^\infty(V)$ be the space of bounded Lebesgue measurable
      functions $f:V\to\R$ so that
      $\norm{f}_{\infty}:=\esssup_{V}\abs{f}<\infty$
      and
      let
      $W^{1,\infty}_{\text{loc}}(V)$ to be the Sobolev
      space of all functions $f:V\to\R$
      so that if $C\subset V$ is
      compact then $f|_C\in L^\infty(C)$,
      the weak partial derivatives $\partial_i f|_C$,
      $i=1,\ldots, n$,
      exist and are such that 
      $\partial_i f|_C\in L^\infty(C)$, see
      \cite[Notation on pages 26 and 36, Definition 4.2]{evans1992measure}.
      Since the coordinate map $\phi$ is smooth, $\tau_S$ is locally Lipschitz
      on $M$ with respect to $h$ if and only if $\tau_S\circ\phi^{-1}$ is
      locally Lipschitz with respect to the Euclidean metric.
      By \cite[Theorem 4.5]{evans1992measure}, $\tau_S\circ\phi^{-1}$ is
      locally Lipschitz  if and only if
      $\tau_S\circ\phi^{-1}\in W^{1,\infty}_{\text{loc}}(\phi(U))$.

      Let $C\subset \phi(U)$ be compact.
      Since $\tau_S\circ\phi^{-1}$ is continuous,
      Lemma 
      \ref{lem_tau_gh_is_cts},
      we see that
      $\tau_S\circ\phi^{-1}$ is Lebesgue measurable
      and $\norm{\tau_S\circ\phi^{-1}|_C}_{\infty}<\infty$.
      That is $\tau_S\circ\phi^{-1}|_C\in L^\infty(C)$.

      Lemma \ref{lem_tau_S_bounded_h_grad}
      implies that there exists ${K_1}\in\R^+$ so that
      wherever the $g$-gradient $\grad\tau_S|_{\phi^{-1}(C)}$ exists we have
      $h(\grad\tau_S|_{\phi^{-1}(C)},\grad\tau_S|_{\phi^{-1}(C)}) < {K_1}^2$.
      
      With $(f_i)_{i=1}^n$ the standard orthonormal basis of $\R^n$, the
      partial derivatives of $\tau_S\circ\phi^{-1}$ are given by
      \[
        \partial_i(\tau_S\circ\phi^{-1})
        =
        D(\tau_S\circ\phi^{-1})\cdot f_i.
      \]
      As $\tau\circ\phi^{-1}$ is continuous, the main result of \cite{MM77}
      implies that the sets where the partial derivatives exist are measurable,
      and the functions $\partial_i(\tau_S\circ\phi^{-1})$ are measurable on
      these sets. 
      As $D(\tau_S\circ\phi^{-1})$ exists a.e., these sets are of
      full measure. As 
      $D(\tau_S\circ\phi^{-1})$
      is bounded by Lemma \ref{lem_tau_S_bounded_h_grad}, so too are the
      partial derivatives and $\tau_S\circ\phi^{-1}\in W^{1,\infty}$. Hence
      $\tau_S\circ\phi^{-1}$ and so $\tau_S$ are locally Lipschitz.
    \end{proof}

    Summarising the results of this section, we have the following.

    \begin{corollary}
      \label{corl:surf}
      Let $(M,g)$ be globally hyperbolic.
      If $S\subset M$ is a $C^1$ Cauchy surface,
      then the surface function $\tau_S$ 
      associated to $S$ is 
      \begin{enumerate}[noitemsep]
        \item a locally anti-Lipschitz, locally Lipschitz, time function,
        \item such that
          $\nabla\tau_S$ exists almost everywhere
          and 
          $g(\nabla\tau_S,\nabla\tau_S)= -1$ wherever $\nabla\tau_S$ exists,
          and
        \item the null distance defined by $\tau_S$ is an actual
          metric which
          induces the manifold topology.
      \end{enumerate}
    \end{corollary}
    \begin{proof}
      Lemma \ref{lem_surface_function_exists} proves that
      any Cauchy surface has an associated surface function.
      That is, the lemma proves that $\tau_S$ is well defined
      by Definition \ref{def_surface_function} for any 
      Cauchy surface.
      Minguzzi, \cite[Theorem 1.19]{minguzzi2019lorentzian}, 
      shows that $\tau_S$ has an almost everywhere defined
      gradient. Proposition 
      \ref{prop_tau_prop} shows that wherever $\nabla\tau_S$
      exists $g(\nabla\tau_S,\nabla\tau_s)= -1$.
      Lemma \ref{lem_tau_gh_is_cts} shows that
      $\tau_S$ is continuous
      and so
      Lemma \ref{lem_gh_sur_gtf} completes the proof that
      $\tau_S$ is a time function.
      Corollary \ref{cor_sf_alip} proves that
      $\tau_S$ is locally anti-Lipschitz.
      We now know \cite[Theorem 4.6]{sormani2016null}
      that the null distance induced by $\tau_S$
      is a distance and induces the manifold topology.
      Lemma \ref{lem_sf_llip}
      proves that $\tau_S$ is locally Lipschitz.
    \end{proof}
    
    \begin{corollary}
    \label{corl:surf2}
    Let $(M,g)$ be globally hyperbolic.
    If $S\subset M$ is a $C^1$ Cauchy surface,
    then the restriction of the surface function $\tau_S$ 
    associated to $S$ to $\future{S}$ is a regular cosmological time function for $\future{S}$. Hence by Theorem \ref{thm:reg-cosmo}, $\tau_S$ restricted to $\future{S}$ satisfies the assumptions of Theorems \ref{lem_cauchy_sort_of} and \ref{thm_reparametrisation_to_control_limsup}. 
    \end{corollary}
    
{\bf Data availability statement} This manuscript has no associated data.    

\pdfbookmark[0]{References}{refs}
\bibliographystyle{plain}
\bibliography{bibliography}

\end{document}